\documentclass[11pt, letterpaper]{article}
\usepackage{amssymb,amsmath,amsthm,tocloft}
\usepackage{fancyhdr}
\usepackage[british]{babel}
\usepackage{enumitem}
\usepackage{lineno}
\usepackage[usenames,dvipsnames,table]{xcolor}
\usepackage{tikz}
\usepackage{mathtools}
\usetikzlibrary{decorations.pathreplacing}
\usepackage{bbm}

\usepackage[bookmarksopen=true,pagebackref=true]{hyperref}
\usepackage{cleveref}

\usepackage{todonotes} 

\usepackage[margin=1in]{geometry}

\newtheorem{theorem}{Theorem}[section]
\newtheorem{prop}[theorem]{Proposition}
\newtheorem{lemma}[theorem]{Lemma}
\newtheorem{cor}[theorem]{Corollary}
\newtheorem{conj}[theorem]{Conjecture}

\newtheorem{thmi}{Theorem}

\newcommand{\repeatlabel}{}
\newtheorem*{repeatlemma}{Lemma \repeatlabel}

\theoremstyle{definition}
\newtheorem{problem}{Problem}

\newtheorem{fact}[theorem]{Fact}
\newtheorem{defn}[theorem]{Definition}

\newtheorem{claim}[theorem]{Claim}

\newcommand{\<}{\subset}

\newcommand{\RT}{{\bf{RT}}}
\newcommand{\RTT}{{\bf{RTT}}}

\newcommand{\cF}{\mathcal{F}}
\newcommand{\cB}{\mathcal{B}}

\newcommand{\De}{\Delta}
\newcommand{\de}{\delta}
\newcommand{\be}{\beta}
\newcommand{\ze}{\zeta}
\newcommand{\ga}{\gamma}
\def \A{\mathcal{A}}
\def \cP{\mathcal{P}}

\def \F{\mathcal{F}}

\def \Z{\mathcal{Z}}
\def \N{\mathbb{N}}
\DeclarePairedDelimiter{\floor}{\lfloor}{\rfloor}
\DeclarePairedDelimiter{\ceil}{\lceil}{\rceil}

\title{A Ramsey--Tur\'an theory for tilings in graphs} 
\author{Jie Han\thanks{School of Mathematics and Statistics, Beijing Institute of Technology, Beijing, China, Email: {\tt hanjie@bit.edu.cn}}, Patrick Morris\thanks{Freie Universit\"at Berlin and Berlin Mathematical School, Germany, Email: {\tt pm0041@mi.fu-berlin.de}.
Research supported by the Deutsche Forschungsgemeinschaft (DFG, German Research
Foundation) under Germany's Excellence Strategy - The Berlin Mathematics
Research Center MATH+ (EXC-2046/1, project ID: 390685689). }, Guanghui Wang\thanks{School of Mathematics, Shandong University, Jinan, China, Email: {\tt ghwang@sdu.edu.cn}. Research supported bye Natural Science
Foundation of China (11871311, 11631014) and seed fund program
for international research cooperation of Shandong University.}, Donglei Yang\thanks{Data Science Institute, Shandong University, Jinan, China, Email: {\tt dlyang@sdu.edu.cn}.}
}

\date{}
\begin{document}
\maketitle

\begin{abstract}

For a $k$-vertex graph $F$ and an $n$-vertex graph $G$, an $F$-tiling in $G$ is a collection of vertex-disjoint copies of $F$ in $G$. For $r\in \N$, the $r$-independence number of $G$, denoted $\alpha_r(G)$, is the largest size of a $K_r$-free set of vertices in $G$. In this paper, we discuss Ramsey--Tur\'an-type theorems for tilings where one is interested in minimum degree and independence number conditions (and the interaction between the two) that guarantee the existence of optimal  $F$-tilings. For cliques, we show that for any $k\geq 3$ and $\eta>0$, any graph $G$ on $n$ vertices with $\delta(G)\ge \eta n$ and $\alpha_k(G)=o(n)$ has a $K_k$-tiling covering all but $\floor{\tfrac{1}{\eta}}(k-1)$ vertices.
All conditions in this result are tight; the number of vertices left uncovered can not be improved and,
for $\eta<\tfrac{1}{k}$, a condition of $\alpha_{k-1}(G)=o(n)$ would not suffice. When $\eta>\tfrac{1}{k}$, we then show that  $\alpha_{k-1}(G)=o(n)$ does suffice, but not $\alpha_{k-2}(G)=o(n)$.  These results unify and generalise previous results of Balogh--Molla--Sharifzadeh~\cite{balog16}, Nenadov--Pehova~\cite{nenadov20} and  Balogh--McDowell--Molla--Mycroft~\cite{balog18} on the subject. We further explore the picture when $F$ is a tree or a cycle and discuss the effect of replacing the independence number condition with  $\alpha^*(G)=o(n)$ (meaning that any pair of disjoint linear sized sets induce an edge between them) where one can force perfect $F$-tilings covering all the vertices. Finally we discuss the consequences of these results in the randomly perturbed graph setting.


\end{abstract}


\section{Introduction}

For a $k$-vertex graph $F$ and an $n$-vertex graph $G$, an \emph{$F$-tiling} is a collection of vertex-disjoint copies of $F$ in $G$. We call an $F$-tiling \emph{perfect} if it covers the vertex set of $G$ (note that if a perfect tiling exists, one must have $n\in k \N$). We will also refer to a perfect $F$-tiling as an \emph{$F$-factor}.  As a natural generalisation of  matchings in a graph (when $F$ is a single edge),  $F$-tilings   are a  fundamental object in 
graph theory with a wealth of results studying various aspects and variants. 
From an extremal perspective, the natural question is as follows.

\begin{problem} \label{p extremal}
  Let $F$ be a graph with $k$ vertices. What minimum degree condition on a graph $G$ with $n\in k\mathbb{N}$ vertices, guarantees the existence of a perfect $F$-tiling? 
\end{problem}

For the special case of cliques, that is when $F=K_k$, the answer to Problem \ref{p extremal} was given by Hajnal and Szemer\'edi \cite{hs} who showed that any $n$-vertex graph $G$ with $\delta(G)\geq \left(1-\tfrac{1}{k}\right)n$ contains a perfect  $K_k$-tiling. 
The case of triangle factors (i.e. $k=3$) was previously shown by Corr\'adi and Hajnal~\cite{corradi1963maximal}. For general $F$, a large body of work culminated in  K\"uhn and Osthus \cite{kuhn09} solving Problem \ref{p extremal} up to an additive constant in the minimum degree condition.

For all these results the necessary minimum degree is large, requiring that $\delta(G)\geq \eta_F v(G)$ for some $\eta_F\geq \tfrac{1}{2}$, and extremal constructions  show that these conditions are necessary. 
However, it is known that such constructions are well-structured and thus far from being typical. For example, in the context of perfect $K_3$-tilings, for which Corr\'adi and Hajnal~\cite{corradi1963maximal} showed that $\eta_{K_3}=\tfrac{2}{3}$, one can take as an extremal graph, a complete tripartite graph with vertex sets of size $\tfrac{n}{3}-1, \tfrac{n}{3},$ and $\tfrac{n}{3}+1$. 
In particular, although this graph and indeed other extremal graphs, are very dense, they contain large independent sets.
This has motivated recent trends in extremal graph theory which have focused on reducing the minimum degree condition necessary by adding an extra condition that provides pseudorandom properties. In this direction, tilings have been explored in popular (strong) notions of pseudorandomness given by so-called bijumbled graphs and $(n,d,\lambda)$-graphs  \cite{HKMP18,HKP18b,krivele04,confversion,Nen18}, as  well as in randomly perturbed graphs\footnote{We discuss this model  in more detail in Section \ref{sec:perturbed}.} \cite{BTW_tilings,bottcher2020triangles,bottcher2021cycle,han2019tilings}, where one is interested in the amount of random perturbation needed to ensure that a dense graph contains a given tiling. Perhaps the weakest pseudorandom condition one can impose on the host graph is to simply block the existence of large independent sets. The following question was proposed by Balogh, Molla and Sharifzadeh \cite{balog16}. 

\begin{problem}\label{p1}
  Let $F$ be a graph with $k$ vertices. What minimum degree on an $n$-vertex graph $G$ with $n\in k \mathbb{N}$ and  $\alpha(G)=o(n)$ guarantees a perfect $F$-tiling? 
\end{problem}

Balogh, Molla and Sharifzadeh also  provided the first result related to Problem \ref{p1} by showing that for any $\eta>\tfrac{1}{2}$, an $n$-vertex graph $G$ with $n\in 3\mathbb{N}$, $\delta(G)\geq \eta n$ and  $\alpha(G)=o(n)$, contains a $K_3$-factor. Their result is tight, as discussed shortly, and shows that adding the condition of small independence number gives a significant reduction in the minimum degree needed to force a triangle factor, from $\tfrac{2n}{3}$, as shown by Corr\'adi and Hajnal~\cite{corradi1963maximal}, to roughly $\tfrac{n}{2}$. More recently, Knierim and Su \cite{knier21} resolved Problem \ref{p1} for all larger cliques $K_k$.
Problem \ref{p1} was inspired by the analogous question for Tur\'an problems where one is interested in the density needed to force the existence of a fixed sized subgraph $H$. Again, the extremal examples are far from typical and  contain large independent sets. Imposing an upper bound on the independence number then leads to improvements on the density needed. This field, known as Ramsey--Tur\'an theory, was initiated by
Erd\H{o}s and S\'{o}s \cite{Erdos70} and led to a wealth of results, see e.g. \cite{sim-sos01} for an overview. One fruitful direction was introduced by Erd\H{o}s, Hajnal, S\'os and Szemer\'edi \cite{Erdos83}. They suggested to strengthen the  condition on independence number slightly, by requiring that $\alpha_r(G)=o(n)$, where $\alpha_r(G)$ denotes the largest size of a $K_r$-free set in $G$. In more detail, for a fixed graph $H$, define\footnote{This function is usually denoted by $\Theta_r(H)$ and  defined (equivalently) in terms of a limit of Ramsey--Tur\'an numbers, see e.g. \cite{balogh2013ramsey}. } the function $\RT_r(H)$ to be the  maximum  $\tau_0\ge0$ such that for any $\tau<\tau_0$ and $\alpha>0$  there exist  $H$-free graphs $G$ with $n$ vertices, $e(G)\ge\tau n^2$ and  $\alpha_{r}(G) \leq \alpha n$ for all sufficiently large $n$. Much research \cite{bollobas1976ramsey,Erdos83,Erdos70,sim-sos01,szemeredi1972graphs} has focused on 
establishing the value of $\RT_r(H)$ for various $r\in \N$ and graphs $H$.

As an analogue of Problem \ref{p1} and taking further inspiration from Ramsey--Tur\'an theory, Nenadov and Pehova \cite{nenadov20} proposed to study the effect of imposing the stronger condition that $\alpha_r(G)=o(n)$. They put forward the following question. 

\begin{problem}\label{p2}
  Let $k \ge r \ge2$ be integers and $F$ a $k$-vertex graph. What minimum degree on an $n$-vertex graph $G$ with  $n\in k \mathbb{N}$ and $\alpha_r(G)=o(n)$ guarantees a perfect $F$-tiling?
\end{problem}

Nenadov and Pehova \cite{nenadov20} proved upper and lower bounds on the minimum degree when $F=K_k$.
In particular, they showed the following generalisation of the aforementioned result of Balogh, Molla and Sharifzadeh \cite{balog16}.

\begin{theorem}\emph{\cite{nenadov20}} \label{thm:Nen-Peh} For any integer  $k\geq 3$ and $\tfrac{1}{2}<\eta\leq 1$  there exists  $\alpha>0$ such that any $n$-vertex graph $G$ with $n\in k\mathbb{N}$, $\de(G)\ge \eta n$ and $\alpha_{k-1}(G)\le \alpha n$ contains a $K_k$-factor. 
\end{theorem}

\subsection{Quasiperfect tilings} \label{sec:quasi}

A prominent theme in Ramsey--Tur\'an theory has been to establish for certain $r$ and $H$, whether 
$\RT_r(H)$ is non-zero, see for instance \cite{balogh2013ramsey,Erdos83,Erdos70,sim-sos01}. In  words, one is interested in whether it is possible to have dense graphs which have copies of $K_r$ in every linear sized set of vertices whilst avoiding a copy of $H$. In this paper, our aim is to study the analogous question for $K_k$-tilings, and more generally, $F$-tilings. We are therefore interested in graphs $G$  which have small linear minimum degree. In this range, however, one can not hope to find $F$-factors. Indeed, for any $\eta<\tfrac{1}{2}$ and $n$ sufficiently large, one can take $G$ to be the union of two disjoint cliques whose sizes add to $n$. Such a $G$ has the property that $\alpha_r(G)< 2r$ for all fixed $r\in \mathbb{N}$ whilst choosing the sizes of the cliques in $G$ appropriately, we can guarantee that $\delta(G)\geq \eta n$ and that any $F$-tiling will leave uncovered vertices. Therefore, we need to relax our expectations, moving away from studying $F$-factors and instead studying tilings which are as large as possible, in that they cover all but a constant number of vertices. 
 Such tilings we will call \emph{quasiperfect tilings}.

\begin{defn} \label{def:quasiperfect}
Given $\eta\in(0,1]$, a  $k$-vertex graph $F$ and an $n$-vertex graph $G$, we say an $F$-tiling in $G$ is \emph{$\eta$-quasiperfect} if it covers all but at most $\ell(k-1)$ vertices of $G$ where $\ell:=\floor{\tfrac{1}{\eta}}$. When $\eta$ is clear from context, we will simply call the tiling quasiperfect.
\end{defn}

Notice that when $\eta>\tfrac{1}{2}$, an $F$-tiling is $\eta$-quasiperfect if it leaves at most $k-1$ vertices uncovered. Therefore, with an additional condition that $n\in k \N$, a quasiperfect tiling is in fact a perfect tiling. For (almost) all values of $\eta$ the definition of quasiperfect tilings captures the largest possible size of a tiling we can hope for when looking at graphs $G$ with $\delta(G)\ge \eta n$ and some bound on the independence number. Indeed, generalising the construction above, if $\tfrac{\delta(G)}{n}=\eta$ and $\tfrac{1}{\eta}\notin \N$, we can take $G$ to be $\ell=\floor{\tfrac{1}{\eta}}$ disjoint cliques of equal size $\tfrac{n}{\ell}$ and choose $n$ such that $\tfrac{n}{\ell}$ is equal to  $(k-1) \mod k$. For a connected $F$, there are no copies of $F$ using more than one clique in $G$ and in each clique  any $F$-tiling must leave  $k-1$ vertices uncovered due to divisibility constraints. 

It turns out that the divisibility constraints given by the construction outlined above are the `worst-case scenario' when we impose an appropriate independence condition. Indeed, our results will show that we can always guarantee a quasiperfect tiling. The first result of this kind was predicted by Alon and proven by Balogh, McDowell, Molla and Mycroft \cite{balog18} who showed the following.

\begin{theorem}\emph{\cite{balog18}}\label{thm:almostk3}
  For every $\eta > \tfrac{1}{3}$ there exists $\alpha>0 $ such that every graph $G$ on $n$ vertices with $\de(G)\ge \eta n$ and $\alpha(G)\le \alpha n$
contains a $K_3$-tiling covering all but at most $4$ vertices.
\end{theorem}

Combined with the case $k=3$ of  Theorem~\ref{thm:Nen-Peh} \footnote{Rather a slightly modified version of Theorem~\ref{thm:Nen-Peh} that does not restrict to the case $n\in 3\mathbb{N}$} proven by Balogh, Molla and Sharifzadeh \cite{balog16} this gives that for any $\eta>\tfrac{1}{3}$, a condition of $\alpha(G)=o(n)$ is enough to guarantee the existence of an $\eta$-quasiperfect $K_3$-tiling in a graph of minimum degree $\eta n$. As we will see,  the value $\tfrac{1}{3}$ can not be improved here, in that for any $\eta<\tfrac{1}{3}$, there are $n$-vertex graphs with $\delta(G)\ge \eta n$ and $\alpha(G)=o(n)$ for which the largest $K_3$-tiling leaves linearly many vertices uncovered. In this sense, $\tfrac{n}{3}$ can be thought of as a minimum degree  threshold for quasiperfect factors. To formalise this and strengthen the  analogy with Ramsey--Tur\'an densities for fixed sized graphs, we make the following definition.

\begin{defn} \label{def:RTT}
For a $k$-vertex graph $F$ and an integer $1\leq r\leq k$, define the \emph{Ramsey--Tur\'an tiling threshold}, denoted by $\RTT_r(F)$, as the largest $\eta_0\ge 0$ such that for all $0<\eta<\eta_0$ and $\alpha>0$, there exist an $n$-vertex graph $G$ with $\delta(G)\ge \eta n$ and $\alpha_r(G)\le \alpha n$ such that $G$ does not contain an $\eta$-quasiperfect $F$-tiling, for all sufficiently large $n$.  
\end{defn}

The parameter $\RTT_r(F)$ asks what density one needs to force quasiperfect $F$-tilings in graphs $G$ with $\alpha_r(G)=o(n)$. In this way, establishing the value of this parameter encompasses many previously studied questions. Indeed Problem \ref{p extremal} can be thought of as the study of $\RTT_1(F)$ as $\alpha_1(G)=0$ for every graph $G$ and so the independence condition becomes obsolete when looking at $\RTT_1$. The problem then reduces to finding extremal minimum degree thresholds, again by noting that when $\eta>\tfrac{1}{2}$ and  $n\in k \N$, quasiperfect tilings reduce to perfect tilings.   Moreover, Problems \ref{p1} and \ref{p2} coincide with finding $\RTT_r(F)$ when $\RTT_r(F)\ge \tfrac{1}{2}$. Finally Theorem \ref{thm:almostk3} can be read as saying that $\RTT_2(K_3)=\tfrac{1}{3}$. We propose a systematic study of the parameter $\RTT_r(F)$.

\subsection{Our results} \label{sec:our results}

Our focus here is to explore small values of $\RTT_r(F)$. We begin by addressing the question of determining when $\RTT_r(F)=0$. Our first result shows that for a $k$-vertex graph $F$, imposing that $\alpha_k(G)=o(n)$ forces any dense graph to contain quasiperfect $F$-tilings.

\begin{theorem} \label{thm:K_k alpha_k}
For any integer $k\ge 3$ and constant $0<\eta\le 1$, there exists an $\alpha>0$ such that for all sufficiently large $n\in\N$, any $n$-vertex graph $G$ with  $\delta(G)\ge \eta n$ and $\alpha_k(G)\le \alpha n$ contains a $\eta$-quasiperfect $K_k$-tiling, and hence an $\eta$-quasiperfect $F$-tiling for any $k$-vertex graph $F$.
\end{theorem}

\Cref{thm:K_k alpha_k} shows that $\RTT_k(F)=0$ for all $k$-vertex graphs $F$. Previously the first and third author together with Treglown (unpublished) used their methods from \cite{han2019tilings} to prove a weaker version of Theorem \ref{thm:K_k alpha_k} (as well as \Cref{thm:K_k alpha_k-1} below),  that did not determine the exact constant number of vertices uncovered by the $K_k$-tiling.
The ideas from \cite{han2019tilings} will also be central to our proofs here.


It is natural to ask whether the condition of $\alpha_k(G)=o(n)$ can be weakened to $\alpha_r(G)=o(n)$ for some $r<k$. It turns out that this is not the case in general, as detailed by the next proposition which follows easily from a construction given by Balogh, Treglown and Wagner \cite[Section 2.1]{BTW_tilings}. We will give the details of the proof here for completeness.

\begin{prop} \label{prop:lower K_k alpha_k-1} For any integer $k\geq 3$, $0<\eta<\tfrac{1}{k}$ and $\alpha>0$ the following holds for all sufficiently large $n\in \N$. 
There exists an $n$-vertex graph $G$ with $\delta(G)\ge \eta n$ and $\alpha_{k-1}(G)\le \alpha n$ such that every $K_k$-tiling in $G$ covers at most $\eta k n$ vertices.
\end{prop}

Proposition \ref{prop:lower K_k alpha_k-1} shows that $\RTT_{k-1}(K_k)\ge \tfrac{1}{k}$ in a strong way. Indeed the constructions have the property that any $K_k$-tiling misses a \emph{linear} number of vertices. Our next result shows that $\RTT_{k-1}(F)\le \tfrac{1}{k}$ for all $k$-vertex graphs $F$, in particular providing a matching upper bound to Proposition \ref{prop:lower K_k alpha_k-1} for $F=K_k$.

\begin{theorem} \label{thm:K_k alpha_k-1}
For any integer $k\ge 3$ and constant $\tfrac{1}{k}<\eta\le 1$, there exists an $\alpha>0$ such that for all sufficiently large $n\in \N$, any $n$-vertex graph $G$ with  $\delta(G)\ge \eta n$ and $\alpha_{k-1}(G)\le \alpha n$ contains an $\eta$-quasiperfect $K_k$-tiling, and hence an $\eta$-quasiperfect $F$-tiling for any $k$-vertex graph $F$.
\end{theorem}

Note that \Cref{thm:K_k alpha_k-1} implies both \Cref{thm:Nen-Peh} and \Cref{thm:almostk3} and bridges the gap between these two theorems whilst generalising \Cref{thm:almostk3} to cliques of all sizes. \Cref{thm:K_k alpha_k-1} is tight for $K_k$ considering Proposition \ref{prop:lower K_k alpha_k-1}.

\subsubsection{Cycles and Trees}

Given that the tightness of \Cref{thm:K_k alpha_k-1} is only witnessed by $F=K_k$, we now explore the behaviour of other graphs $F$. In particular, we can ask whether there are certain graphs $F$ for which the $\alpha_k(G)=o(n)$ condition in \Cref{thm:K_k alpha_k} \emph{can} be weakened to considering independence numbers for smaller cliques. Concentrating on the weakest possible condition of $\alpha(G)=\alpha_2(G)=o(n)$, we now discuss whether such a condition can force quasiperfect $F$-tilings for all dense graphs. Our next two results establish that this is \emph{only} the case if $F$ is  a tree.

\begin{theorem} \label{thm:tree}
For any integer $k\ge 3$ and constant $0<\eta\le 1$, there exists an $\alpha>0$ such that for all sufficiently large $n\in \N$, any $n$-vertex graph $G$ with  $\delta(G)\ge \eta n$ and $\alpha(G)\le \alpha n$ contains a $\eta$-quasiperfect $T_k$-tiling for any $k$-vertex tree $T_k$.
\end{theorem}

For any graph $F$ with a cycle, denote by $\gamma(F)$ the minimum size of a subset $S\< V(F)$ such that $F-S$ is acyclic. Note that $\gamma(K_k)=k-2$ and $\gamma(C_k)=1$ for all $k\geq 3$. 

\begin{prop} \label{prop:lower alpha_2}
For any integer $k\geq 3$, $k$-vertex $F$, $0<\eta<\tfrac{\gamma(F)}{k}$ and $\alpha>0$ the following holds for all sufficiently large $n\in \N$, defining $\mu:=\tfrac{\gamma(F)}{k}-\eta>0$. There exists an $n$-vertex graph $G$ with $\delta(G)\ge \eta n$ and $\alpha(G)\le \alpha n$ such that every $F$-tiling in $G$ covers at most $(1-\mu)n$ vertices.
\end{prop}

Proposition \ref{prop:lower alpha_2} proves that\footnote{If $F$ is acyclic, this is immediate.} $\RTT_{2}(F)\geq \tfrac{\gamma(F)}{k}$ for every $k$-vertex $F$. In particular, we have that $\RTT_2(K_k)\geq 1-\tfrac{2}{k}$. This was previously shown by Knierim and Su \cite{knier21} who proved a matching upper bound (restricted to the case $n\in k \N$ and hence concentrating on $K_k$-factors), as previously mentioned. Our next result proves that this lower bound is tight for cycles and can be seen as a generalisation of \Cref{thm:almostk3} to cycles of arbitrary length.

\begin{theorem} \label{thm:cycle}
For any integer $k\ge 3$ and constant $\tfrac{1}{k}<\eta\le 1$, there exists an $\alpha>0$ such that for all sufficiently large $n\in \N$, any $n$-vertex graph $G$ with  $\delta(G)\ge \eta n$ and $\alpha(G)\le \alpha n$ contains an $\eta$-quasiperfect $C_k$-tiling, where $C_k$ denotes the cycle with $k$ vertices.
\end{theorem}

We summarise all our results on Ramsey--Tur\'an tiling thresholds in the following theorem.

\begin{thmi}[$\RTT$ bounds] \label{thm:RTTeverything}
The following statements hold for all $k\in\N$ with $k\ge 3$:
\begin{enumerate}[label=$(\roman*)$]
    \item  \label{RTTi} $\RTT_k(F)=0$ for any $k$-vertex $F$;
    \item \label{RTTii} $\RTT_{k-1}(K_k)\geq \tfrac{1}{k}$;
    \item \label{RTTiii} $\RTT_{k-1}(F)\le \tfrac{1}{k}$ for all $k$-vertex graphs $F$;
    \item \label{RTTiv} $\RTT_2(T)=0$ for any tree $T$;
    \item \label{RTTv} $\RTT_{2}(F)\geq \tfrac{\gamma(F)}{k}$  for any $k$-vertex $F$ where $\gamma(F)$ is defined as the minimum size of a subset $S\subset V(F)$ such that $F- S$ is acyclic;
    \item \label{RTTvi} $\RTT_{2}(C_k)\leq \tfrac{1}{k}$ for any $k$-vertex cycle $C_k$.
\end{enumerate}
\end{thmi}

Note that Theorem \ref{thm:RTTeverything} \ref{RTTii} and \ref{RTTiii} imply that $\RTT_{k-1}(K_k)=\tfrac{1}{k}$ and, similarly, parts \ref{RTTv} and \ref{RTTvi} determine that $\RTT_2(C_k)=\tfrac{1}{k}$. For ease of reference, from now on we will work solely with the statements in Theorem \ref{thm:RTTeverything} but for more detailed statements, we note here that \ref{RTTi} is Theorem \ref{thm:K_k alpha_k}, \ref{RTTii} is Proposition \ref{prop:lower K_k alpha_k-1}, \ref{RTTiii} is Theorem \ref{thm:K_k alpha_k-1}, \ref{RTTiv} is Theorem \ref{thm:tree}, \ref{RTTv} is Proposition \ref{prop:lower alpha_2} and  \ref{RTTvi} is Theorem \ref{thm:cycle}.

\subsection{Removing divisibility barriers by forbidding large holes}

Keevash and Mycroft \cite{keevash15} developed a geometric theory for tilings showing that obstructions for perfect tilings come in the form of either \emph{divisibility} or \emph{space} `barriers'  (see their work for precise definitions of these concepts). Our results above essentially show that by imposing appropriate minimum degree and  independence  conditions on the host graph, one can remove the space barriers and force an $F$-tiling covering all but a constant number of vertices. As we outlined in Section \ref{sec:quasi}, when the minimum degree is less than $\frac{n}{2}$, imposing bounds on  independence numbers can not remove divisibility barriers and so the largest tilings we can guarantee are quasiperfect.

Various authors have suggested imposing slightly stronger conditions in order to overcome the divisibility barriers and prove the existence of perfect tilings.  Balogh, McDowell, Molla and Mycroft~\cite{balog18} propose looking at graphs which are $K_r$-free for some $r>k$ whilst Nenadov and Pehova \cite{nenadov20} suggest to strengthen the independence condition by forbidding large partite `holes'.

\begin{defn} \label{def:holes}
For an integer $r\ge 2$, an $r$-partite hole of size $s$ in a graph $G$ is a collection of $r$ disjoint vertex subsets $U_1,\ldots,U_r\subset V(G)$ of size $s$ such that there is no copy of $K_r$ in $G$ with one vertex in each of the $U_i$, $i\in[r]$.
For a graph $G$, let $\alpha^*_r(G)$ denote the size of the largest $r$-partite hole in $G$. When $r=2$, we also refer to bipartite holes and we sometimes drop the subscript, letting $\alpha^*(G):=\alpha_2^*(G)$.
\end{defn}

The notion of bipartite holes was first introduced by McDiarmid and Yolov \cite{mcd17} whilst studying Hamiltonian cycles.
Note that $\floor{\tfrac{\alpha_r(G)}{r}}\le  \alpha_r^*(G)$ for all graphs $G$ and so imposing that $\alpha_r^*(G)=o(n)$ is a stronger assumption. Moreover, for any $r\ge 2$, forbidding large $r$-partite holes precludes the existence of the disjoint cliques constructions discussed in \Cref{sec:quasi}. Therefore one can expect to obtain perfect tilings with such a condition and our results will show that this is indeed the case. We define the following variation of Definition~\ref{def:RTT}.

\begin{defn} \label{def:RTT*}
For a $k$-vertex graph $F$ and an integer $1\leq r\leq k$, let $\RTT^*_r(F)$ denote the largest $\eta^*_0\ge 0$ such that for all $0<\eta<\eta^*_0$ and $\alpha>0$, there exists  $n$-vertex graphs $G$ with $\delta(G)\ge \eta n$ and $\alpha^*_r(G)\le \alpha n$ such that $G$ does not contain a perfect $F$-tiling, for all sufficiently large $n\in k \N$.  
\end{defn}

All the results on $\RTT_r(F)$ in Theorem~\ref{thm:RTTeverything}, except \ref{RTTiii} and \ref{RTTvi} when $k=3$, transfer to the setting of $\RTT^*_r(F)$. That is, in Theorem~\ref{thm:RTTeverything} \ref{RTTv} we can replace the condition $\alpha_{2}(G)\le \alpha n$ with $\alpha^*_2(G)\le \alpha n$ and in Theorem~\ref{thm:RTTeverything} \ref{RTTi}, \ref{RTTiv} and \ref{RTTvi} (excluding $k=3$), replacing the condition  $\alpha_r(G)\le \alpha n$ with $\alpha^*_r(G)\le \alpha n$ strengthens the conclusion by giving an $F$-tiling covering all but $k-1$ vertices and hence an $F$-factor when $n\in k \N$. The variation of  Theorem~\ref{thm:RTTeverything} \ref{RTTi} showing that $\alpha_k^*(G)=o(n)$ guarantees that any dense graph with $n\in k\N$ vertices has an $K_k$-factor (i.e. that $\RTT_k^*(K_k)=0$), was shown previously by Nenadov and Pehova \cite{nenadov20}. All the further generalisations of Theorem \ref{thm:RTTeverything}   to the setting of  $\RTT^*_r(F)$ discussed above will be proven here and are collected below in Theorem \ref{thm:stareverything}.

One might expect from these results that the two parameters $\RTT$ and $\RTT^*$ are equal for all inputs. Interestingly though, this is not the case.
Indeed, Theorem~\ref{thm:RTTeverything} \ref{RTTiii} does not transfer to the setting of $\RTT^*_r(F)$ and in  particular, the bound on minimum degree condition in  \ref{RTTvi} of Theorem~\ref{thm:RTTeverything} is not true for $k=3$ in the context of $\RTT^*$. In fact, we determine that $\RTT^*_2(K_3)=\frac{1}{2}$. We collect all our results on $\RTT^*$ in Theorem \ref{thm:stareverything}.


\begin{thmi}[$\RTT^*$ bounds] \label{thm:stareverything}
The following statements hold for $k\in \N$ with $k\ge 3$:
\begin{enumerate}[label={(\arabic*)}]
    \item   \label{RTT*1} $\RTT^*_2(T)=0$ for any tree $T$;
    \item \label{RTT*2} $\RTT^*_{2}(F)\geq \tfrac{\gamma(F)}{k}$  for any $k$-vertex $F$ where $\gamma(F)$ is defined as the minimum size of a subset $S\subset V(F)$ such that $F- S$ is acyclic;
    \item \label{RTT*3} $\RTT^*_{2}(C_k)\leq \tfrac{1}{k}$ for any $k$-vertex cycle $C_k$ with $k\ge4$;
    \item \label{RTT*4} $\RTT^*_2(K_3)\geq \frac{1}{2}$.
\end{enumerate}
\end{thmi}

Note that parts \ref{RTT*2} and \ref{RTT*3} determine that $\RTT_2^*(C_k)=\tfrac{1}{k}$ for $k\geq 4$. For $k=3$, we have that part \ref{RTT*4} of Theorem \ref{thm:stareverything} along with the $K_3$ case of  \Cref{thm:Nen-Peh} proven in \cite{balog16}, imply that $\RTT^*_2(K_3)=\tfrac{1}{2}$. We conclude the introduction by discussing applications of our results to randomly perturbed graphs.

\subsection{Randomly perturbed graphs} \label{sec:perturbed}

The study of Ramsey--Tur\'an theory for tilings contributes to the greater aim of understanding the barriers that force the extremal tiling thresholds for dense graphs (see Problem~\ref{p extremal}) to be large. A recent trend in this direction has been to study the effect of  small random perturbations on dense graphs. This was initiated by Bohman, Frieze and Martin \cite{bohman2004adding} who studied Hamilton cycles. For tilings, a series of papers \cite{BTW_tilings,bottcher2020triangles,bottcher2021cycle,han2019tilings}  have explored the relationship on the minimum degree of a graph $G$ and the value of $p$ such that $G\cup G(n,p)$ contains a given tiling with high probability (that is, with probability tending to $1$ as $n$ tends to infinity). Given that the conditions  $\alpha_r(G)=o(n)$, and more pertinently $\alpha_r^*(G)=o(n)$, are typical in sparse graphs of a  certain density, our results have implications for the randomly perturbed model. These corollaries are often best possible as one needs the random graph to provide small independence numbers in order to give the existence of tilings in the perturbed model. Moreover, as noticed by Nenadov and Pehova \cite{nenadov20}, one in fact obtains something stronger. Indeed, usually in the perturbed model, one fixes an arbitrary graph $G$ (which satisfies a dense minimum degree condition) and asks for $p$ such that $G\cup G(n,p)$ contains a given tiling with high probability. Here, we can conclude that with high probability, $G(n,p)$ has the property that no matter how an adversary places a graph $G$ (satisfying a minimum degree condition), the resulting graph will have a given tiling.  Below, we collect these corollaries, both of which are tight with respect to the minimum degree and probability conditions.

\begin{cor} \label{cor:perturbed}
Let $k\in \N$ and $n\in k\N$ with $k\ge4$. With high probability $G(n,p)$ has the property that for any graph $G$ such that $\delta(G)\ge \eta n$, $G\cup G(n,p)$ contains an $F$-factor in the following cases:
\begin{enumerate}[label={$(\arabic*)$}]
    \item \label{RP:3}  $F=T_k$ for some tree $T_k$, $\eta>0$ and $p\ge Cn^{-1}$ for $C=C(\eta)$;
    \item \label{RP:4}  $F=C_k$, $\eta>\tfrac{1}{k}$ and $p\ge Cn^{-1}$ for $C=C(\eta)$.
\end{enumerate}
\end{cor}

Corollary \ref{cor:perturbed} follows simply from the relevant parts of \Cref{thm:stareverything} and the fact\footnote{This is a standard application of Janson's inequality.} that for any $\alpha>0$, there exists  $C>0$ such that $\alpha^*_r(G(n,p))\le \alpha n$ with high probability whenever $p\ge C n^{-2/r}$. An analogue of Corollary \ref{cor:perturbed} \ref{RP:4} in the perturbed setting (when $G$ is fixed, as discussed above) 
was recently given by B\"ottcher, Parczyk, Sgueglia and Skokan~\cite{bottcher2021cycle} and a construction detailing the tightness can be found there.
Part \ref{RP:3} of Corollary \ref{cor:perturbed} is in fact not new -- it is implied by a result of B\"ottcher, Kohayakawa, Montgomery, Parczyk, Person and the first author~\cite[Theorem 2]{Bottcher19} on embedding bounded degree spanning trees.

\subsection*{Notation}
Given a graph $G=(V,E)$ we let $v(G)=|V|$ and $e(G)=|E|$. For $U\<V$, $G[U]$ denotes the induced graph of $G$ on $U$. The notation $G-U$ is used to denote the induced graph after removing $U$, that is $G-U:=G[V\setminus U]$. For two subsets $A,B\<V(G)$, we use $e(A,B)$ to denote the number of edges with one endpoint in $A$ and another in $B$, where any possible edge in $G[A\cap B]$ is counted twice. For any $u,v\in V(G)$, we write $N(u,v)$ for all the common neighbours of $u,v$ and $d_A(v)$ for the number of neighbours of $v$ that lie inside $A$. When $A=V$, we drop the subscript and simply write $d(v)$.  Sometimes, we  will use the notation $d^G(v)$ to emphasise that we look at the degree of $v$ relative to the graph $G$.

At times we have statements such as the following. Choose constants $0< c_1 \ll c_2 \ll \ldots \ll c_k= c.$  This should be taken to mean that one can choose constants from right to left so that all the subsequent constraints are satisfied. That is, given some $c>0$ there exist increasing functions $f_i$ for $i\in [k]$ such that whenever $c_{i}\leq f_{i+1}(c_{i+1})$ for all $i\in[k-1]$, all constraints on these constants that are in the proof, are satisfied. We say an event in a probability space holds asymptotically almost surely (and abbreviate this to a.a.s.) when the probability it holds tends to $1$ as some parameter $n$ (usually the number of vertices) tends to infinity. Finally we ignore floors and ceilings where possible, so as not to clutter the exposition.

\section{Proof strategy and organisation of the paper}\label{2}

Recall that all the statements that we will prove are collected in Theorems \ref{thm:RTTeverything} and \ref{thm:stareverything}. We will prove all our upper bounds on $\RTT$ and $\RTT^*$ in the same general framework which will comprise the majority of the paper, spanning Sections~\ref{sec:Abs sets}, \ref{sec:lattice} and \ref{sec:almost} as well as   this section in which we give an overview of  these proofs.
The constructions which give lower bounds on $\RTT$ and $\RTT^*$ (Theorem \ref{thm:RTTeverything}~\ref{RTTii}, \ref{RTTv} and Theorem \ref{thm:stareverything} \ref{RTT*2}, \ref{RTT*4}) are then discussed in \Cref{5}. In \Cref{sec:conclude} we discuss directions for future research. We now discuss our proof strategy for the upper bounds.   

We shall give a unified approach for Theorem \ref{thm:RTTeverything} \ref{RTTi}, \ref{RTTiii}, \ref{RTTiv} and \ref{RTTvi} as well as Theorem \ref{thm:stareverything} \ref{RTT*1} and \ref{RTT*3}. Our proof uses the the absorption method, which has appeared in various guises since the 90s and was widely popularised by R\"{o}dl, Ruci\'{n}ski and Szemer\'{e}di \cite{rodl09} about a
decade ago. In recent years, the method has become an extremely important tool for studying the existence of spanning structures in graphs, digraphs and hypergraphs and novel variations of the method have been developed to overcome certain challenges. Our proofs here build on several of these recent innovations. In particular, we will use techniques developed  in \cite{han16,han17,keevash15} (the so-called lattice based absorption method) as well as a powerful method introduced by Montgomery~\cite{montgo19}   that uses bipartite template graphs with robust matching properties to define absorbing structures.   

In what follows we sketch our proofs, introduce some key concepts and reduce the proofs of our upper bounds to three main propositions (namely  Propositions~\ref{prop:Absorbing sets star}, \ref{prop:Absorbing sets} and~\ref{prop:almost factors}). Throughout this discussion (and indeed the proof), we think of  $G$ as an  $n$-vertex (host) graph and $F$ as a $k$-vertex graph, with the aim of finding an $F$-tiling in $G$.  Most of our definitions will be relative to both $G$ and $F$ but we drop this dependence, thinking of both graphs as being fixed (which most of the time they will be).   

The general idea of absorption is to split the problem of finding a (quasi-)spanning structure into two subproblems. The first major task is to define and find an absorbing structure in the host graph which can `absorb' left-over vertices. We will use the following notion of an absorbing set as used, for example, in \cite{nenadov20}.
\begin{defn} \label{def:absorbing set simple}
        A subset $A\subset V(G)$ is a $\xi$-\emph{absorbing set}  for some constant $\xi>0$ if for any subset $U\subset V(G)\setminus A$ of size at most $\xi n$ such that $k$ divides $|A\cup U|$, the graph $G[A\cup U]$ contains an $F$-factor.
\end{defn}
Let us first focus on the setting of Theorem~\ref{thm:stareverything}, which is easier as we can avoid many technicalities. Here, we have the following key proposition which gives the existence of an absorbing set.

\begin{prop}[Existence of absorbing set] \label{prop:Absorbing sets star}
Fix $k\in \N$ with $k\ge 3$, a $k$-vertex graph $F$ and  $\eta>0$ such that one of the following holds:
\begin{enumerate}
  \item [$(1^*)$] $F=T_k$ for some $k$-vertex tree and $\eta>0$;
  \item [$(2^*)$] $F=C_k$, $k\geq 4$ and $\eta\geq\frac{1}{k}$.
\end{enumerate}
Then for any $\gamma,\mu>0$, there exist  $\alpha,\xi>0$ such that the following holds for all sufficiently large $n\in \N$. If $G$ is an $n$-vertex graph with $\delta(G)\geq (\eta+\mu) n$ and  $\alpha_2^*(G)\leq \alpha n$, then $G$ contains a $\xi$-absorbing set of size at most $\gamma n$.
\end{prop}

Unfortunately, in our general setting of studying $\RTT$, that is, with conditions bounding $\alpha_r(G)$ as opposed to $\alpha_r^*(G)$,  it is possible that absorbing sets do not exist in our host graphs. Indeed, this is the case when we have disconnected components as divisibility issues could arise. We instead have to settle for absorbing sets that can only absorb vertices from a part of the graph. This gives rise to the following definition.

\begin{defn}\label{newdef}
For any subset $V'\subset V(G)$, we say a subset $A\subset V(G)$ is a $\xi$-\emph{absorbing set with respect to} $V'$ for some constant $\xi>0$ if for any subset $U\subset V'\setminus A$ of size at most $\xi n$ such that $k$ divides $|A\cup U|$, the graph $G[A\cup U]$ contains an $F$-factor.
\end{defn}
Unless otherwise stated, we always take $V'=V(G)$ and then simply call $A$ a $\xi$-absorbing set, as in Definition \ref{def:absorbing set simple}. Our next major proposition shows that when bounding $\alpha_r(G)$ (that is, in the context of Theorem \ref{thm:RTTeverything}), we can partition our vertex set into parts, each of which has an absorbing set associated to it.

\begin{prop}[Existence of absorbing set partition] \label{prop:Absorbing sets}
Fix $k,r\in \N$ with $k\geq 3$, a $k$-vertex graph $F$, and  $\eta>0$ such that one of the following holds:
\begin{enumerate}
  \item [$(1)$] $F=K_k$, $r=k$ and $\eta>0$;
  \item [$(2)$] $F=K_k$, $r=k-1$ and $\eta\ge\frac{1}{k}$;
  \item [$(3)$] $F=T_k$ for some $k$-vertex tree, $r=2$ and $\eta>0$;
  \item [$(4)$] $F=C_k$, $r=2$ and $\eta\ge\frac{1}{k}$.
\end{enumerate}
Then for any $\gamma,\mu>0$, there exist $\alpha, \xi>0$ such that the following holds for all sufficiently large $n\in \N$. If $G$ is an $n$-vertex graph with $\delta(G)\geq (\eta+\mu)n$ and $\alpha_r(G)\leq \alpha n$, then there is a partition
\[\cP=\{V_1,\ldots,V_C,A_1,\ldots,A_C,B\}\]
of $V(G)$ for some $C\leq \floor{\tfrac{1}{\eta}}$ such that $|A_1\cup\ldots\cup A_C\cup B|\leq \gamma n$,  $A_i$ is a $\xi$-absorbing set with respect to $V_i$ for each $i\in [C]$, and $G[B]$ has a perfect $F$-tiling.
\end{prop}

Propositions \ref{prop:Absorbing sets star} and \ref{prop:Absorbing sets} deal with finding absorbing structures in our host graph, which will be used to finish our $F$-tilings. The second major task in  absorption arguments for tilings is to find a tiling that covers most of the vertices, leaving just a small linear number of vertices uncovered. We will sometimes call such a tiling an \emph{almost perfect tiling} or an \emph{almost factor}. Our final major proposition shows that in all the settings we are interested in, almost perfect tilings do indeed exist.

\begin{prop}[Existence of almost perfect tilings] \label{prop:almost factors}
Fix $r,k\in \N$ with $k\geq 3$, a $k$-vertex graph $F$ and  $\eta>0$ such that one of the following holds:
\begin{enumerate}
   \item [$(1)$] $F=K_k$, $r=k$ and $\eta>0$;
  \item [$(2)$] $F=K_k$, $r=k-1$ and $\eta\ge\frac{1}{k}$;
  \item [$(3)$] $F=T_k$ for some $k$-vertex tree, $r=2$ and $\eta>0$;
  \item [$(4)$] $F=C_k$, $r=2$ and $\eta\ge\frac{1}{k}$.
\end{enumerate}
Then for any $\delta, \mu>0$ there exists an $\alpha>0$ such that the following holds for all sufficiently large $n\in \N$. If $G$ is an $n$-vertex graph with $\delta(G)\geq (\eta+\mu)n$ and 
$\alpha_r(G)\leq \alpha n$, then $G$ contains an $F$-tiling covering all but $\delta n$ vertices.

\end{prop}

With these propositions in hand we can easily prove our upper bounds on $\RTT$ and $\RTT^*$. We begin with the upper bounds in Theorem~\ref{thm:stareverything}.

\begin{proof}[Proof of Theorem \ref{thm:stareverything} \ref{RTT*1} and \ref{RTT*3}]
The proofs of \ref{RTT*1} and \ref{RTT*3} are essentially identical and so we only prove part \ref{RTT*3}. Fixing some $k\in \N$ with $k\geq 4$, $\eta=\tfrac{1}{k}$ and some $\mu>0$, choose $0< \alpha\ll\delta\ll\xi\ll \gamma\ll\mu$. 
It suffices to show that for all sufficiently large $n\in k\N$, any graph $G$ with $n$ vertices, $\delta(G)\geq (\eta+\mu)n$ and $\alpha_2^*(G)\leq \alpha n$ contains a perfect $C_k$-tiling. So fix such a graph $G$.

Proposition~\ref{prop:Absorbing sets star} implies that  $G$ contains a $\xi$-absorbing set $A$ of size at most $\ga n$. Let $G'=G-A$ and note that  $\de(G')\ge(\eta+ \tfrac{\mu}{2})v(G') $ due to the fact that $\gamma\ll \mu$. By applying Proposition~\ref{prop:almost factors} on $G'$ (and noting that $\alpha_2(G)\le 2\alpha_2^*(G)$), we obtain a $C_k$-tiling $\mathcal{M}$ that covers all but a set $U$ of at most $\de n$ vertices in $G'$. By the absorbing property of $A$, $G[A\cup U]$ contains a $C_k$-factor $\mathcal{R}$, which together with $\mathcal{M}$ forms a $C_k$-factor in $G$. Note that we used that $|A\cup U|=n-k|\mathcal{M}|$ is divisible by $k$ here as well as the fact that $\delta\ll \xi$.
\end{proof}

The proof of the upper bounds for Ramsey--Tur\'an tiling thresholds also follow easily from Propositions~\ref{prop:Absorbing sets} and \ref{prop:almost factors}, with some minor technicalities.

\begin{proof}[Proofs of Theorem \ref{thm:RTTeverything} \ref{RTTi}, \ref{RTTiii}, \ref{RTTiv} and \ref{RTTvi}]

Consider parameters $r,k\in\N$, a $k$-vertex graph $F$ and  $\eta>0$  so that one of the following holds:
\begin{enumerate}
  \item [$(1')$] $F=K_k$, $r=k$;
  \item [$(2')$] $F=K_k$, $r=k-1$ and $\eta>\frac{1}{k}$;
  \item [$(3')$] $F=T_k$, $r=2$ ;
  \item [$(4')$] $F=C_k$, $r=2$ and $\eta>\frac{1}{k}$.
\end{enumerate}
In order to prove all the relevant theorems, it suffices to show that for any such choice of $r,k,F$ and $\eta>0$, there exists an $\alpha>0$ such that for all sufficiently large $n$, any $n$-vertex graph $G$ with $\delta(G)\geq \eta n$ and $\alpha_r(G)\leq \alpha n$ contains an $\eta$-quasiperfect $F$-tiling. Also note that cases $(1')-(4')$ here are precisely
cases $(1)-(4)$ in Propositions~\ref{prop:Absorbing sets} and \ref{prop:almost factors} with the slight tweak that in cases  $(2')$ and $(4')$ we insist that the constant $\eta$ is strictly larger than $\tfrac{1}{k}$.

So let us fix some arbitrary choice of $r,k,F$ and $\eta$ satisfying one of $(1')-(4')$. Let $\{\tfrac{1}{\eta}\}=\tfrac{1}{\eta}-\floor{\tfrac{1}{\eta}}$ be the fractional part of $\tfrac{1}{\eta}$ and $\eta'>0$ be defined such that $\tfrac{1}{\eta'}=\tfrac{1}{\eta}+\tfrac{1}{2}(1-\{\tfrac{1}{\eta}\})$. Then by letting $\mu=\eta-\eta'$, it is easy to check that $0<\mu<\eta$ and $\lfloor\tfrac{1}{\eta}\rfloor=\lfloor\tfrac{1}{\eta'}\rfloor$. Now we have that  it suffices to find an $F$-tiling  covering all but $\lfloor\tfrac{1}{\eta'}\rfloor(k-1)$ vertices of our host graph $G$.

Choose $0< \alpha\ll\delta\ll\xi\ll \gamma\ll\eta',\mu$ and fix some $n$-vertex graph $G$ with $n$ sufficiently large,  $\delta(G)\ge (\eta'+\mu)n=\eta n$ and $\alpha_r(G)\le \alpha n$.  By Proposition~\ref{prop:Absorbing sets} (applied with $\eta'$ playing the role of $\eta$), we have that there exists a partition
\[\cP=\{V_1,\ldots,V_C,A_1,\ldots,A_C,B\}\]
of $V(G)$ for some $C\leq \floor{\tfrac{1}{\eta'}}$ such that $|A_1\cup\ldots\cup A_C\cup B|\leq \gamma n$,  $A_i$ is a $\xi$-absorbing set with respect to $V_i$ for each $i\in [C]$, and $G[B]$ has a perfect $F$-tiling. Let $\cB$ denote the $F$-tiling covering the vertices in $B$ and let  $R=\left(\bigcup_{i\in[C]}A_i\right)\cup B$.
Furthermore, let $G'=G-R$. Then due to our choice of constants,
\[\de(G')\ge \left(\eta' +\tfrac{\mu}{2}\right) n\geq \left(\eta' +\tfrac{\mu}{2}\right) v(G').\]
Therefore due to Proposition~\ref{prop:almost factors}, we obtain an $F$-tiling $\mathcal{M}$ that covers all but a set $U$ of at most $\de n$ vertices in $G'$. Now by the absorbing property of each $A_i$, for any subset $U_i\subset U\cap V_i$ with $|A_i\cup U_i|\in k\mathbb{N}$ and $|U_i|\leq \xi n$, $G[A_i\cup U_i]$ contains an $F$-factor.
Thus we can remove at most $k-1$ vertices from each $U_i$ to get a set satisfying $|A_i\cup U_i|\in k\mathbb{N}$ for each $i\in [C]$ and hence an $F$-tiling $\cF_i$ covering $A_i\cup U_i$, using that $\de\ll \xi$.
Combining the $\cF_i$, $\mathcal{B}$ and $\mathcal{M}$, we obtain an $F$-tiling covering all but at most $C(k-1)$ vertices in $G$.
This completes the proof, noting that $C\le \floor{\tfrac{1}{\eta}}$.
\end{proof}

In order to prove our upper bounds then, it suffices to prove Propositions~\ref{prop:Absorbing sets star}, \ref{prop:Absorbing sets} and \ref{prop:almost factors}. We first deal with finding absorbing sets. In Section~\ref{sec:Abs sets}, we show how we can derive the existence of absorbing sets by showing the existence of many smaller structures which we call absorbers and we reduce Propositions~\ref{prop:Absorbing sets star} and \ref{prop:Absorbing sets} to Lemmas~\ref{lem:detectpartition} and \ref{lem:closed}. In Section~\ref{sec:lattice} we then prove these lemmas, showing how to partition the vertex set of the host graph to obtain sets from which we can build our absorbing sets. Section~\ref{sec:almost} is then devoted to proving the existence of almost perfect tilings and proving Proposition~\ref{prop:almost factors}. In fact, parts of Proposition~\ref{prop:almost factors} are immediate from our conditions but we defer this discussion to Section \ref{sec:almost}.

\section{Absorbing sets} \label{sec:Abs sets}

In this section, we discuss how to define and find absorbing sets in our host graph. We define the following key notion of  \emph{absorbers} (following the notation in \cite{nenadov20}) which we will used as `building blocks' to build absorbing structures.
As in the previous section, we let $G$ be an $n$-vertex graph and $F$ be a $k$-vertex graph throughout this section and think of these as both being fixed in our definitions.
\begin{defn} \label{def:absorber}
For any $S\in \binom{V(G)}{k}$ and an integer $t$, we say a subset $A_S\subset V(G)\setminus S$ is an $(F,t)$-\emph{absorber} for $S$ if $|A_S|\le k^2t$ and both $G[A_S]$ and $G[A_S\cup S]$ contain an $F$-factor.
\end{defn}

Such constant sized sets  are widely used in order to build absorbing structures. However many of these constructions of absorbing sets, for example by R\"{o}dl, Ruci\'{n}ski and Szemer\'{e}di \cite{rodl09}
and H\`{a}n, Person, and Schacht \cite{han09}, rely on the property that every $k$-subset in $V(G)$ has polynomially many   absorbers of a certain type (e.g.\ $\Omega(n^a)$ absorbers $A$ with $v(A)=a$). In our case, as pointed out in \cite{balog16}, the degree conditions in our results are not strong enough to guarantee this property.  We therefore make use of a new construction which guarantees a $\xi$-absorbing set provided that \emph{every} $k$-set $S$ has \emph{linearly many vertex-disjoint $(F,t)$-absorbers}. The key idea that makes this possible stems back to Montgomery~\cite{montgo19} and has since found many applications in absorption arguments. Here, we will follow  the  approach (and notation) of  Nenadov and Pehova \cite{nenadov20} who proved\footnote{To be  exact, the lemma of Nenadov and Pehova \cite[Lemma 2.2]{nenadov20} differs slightly from ours here as they define absorbers to have at most $kt$ vertices whereas we define the upper bound to be $k^2t$. This minor adjustment has no bearing on the statement of this lemma or its proof.} the following.

\begin{lemma}\emph{\cite[Lemma 2.2]{nenadov20}}\label{bip-temp}
Let $k,t\in \N$, $F$  a graph on $k$ vertices and  $\gamma >0$ . Then there exists $\xi = \xi(k,t,\gamma)>0$ such that  the following holds. If $G$ is an $n$-vertex graph such that for
every $S\in \binom{V(G)}{k}$ there is a family of at least
$\gamma n$ vertex-disjoint $(F,t$)-absorbers, then $G$ contains a $\xi$-absorbing set of size at most $\gamma n$.
\end{lemma}

Lemma \ref{bip-temp} is applicable in the setting of bounding large bipartite holes and will be applied in the proof of Theorem \ref{thm:stareverything} \ref{RTT*1} and \ref{RTT*3}. However, in general there might be $k$-sets with few  absorbers or even none at all\footnote{This is the case when the host graph consists of a constant number of cliques and the $k$-set has all but one vertex in the same clique, for example.}.
 To overcome this issue, we develop a generalisation of Lemma \ref{bip-temp} in the following, which provides a family of pairwise disjoint absorbing sets.

 \subsection{Building absorbing sets} \label{sec:building absorbing sets}

 Before introducing our key lemma that proves the existence of absorbing sets given the existence of many local structures, we need to introduce a further simple notion.

 \begin{defn} \label{def:fan}
 For a vertex $v\in V(G)$ and subset of vertices $U\subset V(G)$,  an $F$-fan $\cF_v$ at $v$ in $U\subset V(G)$ is a collection of pairwise disjoint sets $S\subset U\setminus \{v\}$ such that for each $S\in \cF_v$, we have that $|S|=k-1$ and  $\{v\}\cup S$ spans a copy of $F$. If $U=V(G)$, we simply refer to an $F$-fan at $v$.  We further define the vertex set $V(\cF_v)=\cup_{S\in \cF_v}S$  of a fan to be the union of the sets contained in $\cF_v$   and the \emph{size} of $\cF_v$ as the number of sets in $\cF_v$.
 \end{defn}

 Note that a fan of size $\ell$ has precisely $\ell(k-1)$ vertices, recalling that the sets in a fan are pairwise disjoint.
In order to build an absorbing structure, we will also make use of an auxiliary `bipartite template'.  The following lemma (in a slightly different form)  was first introduced by Montgomery \cite{montgo19}. Here we state the form given by Nenadov and Pehova  \cite{nenadov20}.

\begin{lemma}
\emph{ \cite[Lemma 2.8]{montgo19}; \cite[Lemma 2.3]{nenadov20}}\label{robust}
Let~$\be>0$ be given.  There exists $m_0$ such that the following holds for every $m\geq m_0$. There exists a bipartite graph $B$ with vertex classes $X_m\cup Y_m$ and $Z_m$ and maximum degree $\De(B)\leq40$, such that $|X_m|=m+\be m$, $|Y_m| =2m$ and $|Z_m| = 3m$, and for every subset $X_m'\subseteq X_m$ with $|X_m'|=m$, the induced graph $B[X_m'\cup Y_m, Z_m]$ contains a perfect matching.
\end{lemma}

 We can now turn to stating and proving our generalisation of Lemma~\ref{bip-temp}.



\begin{lemma}\label{bip-temp+}
Given any constant $\ga>0$ and  $k,t, C\in \N$ with $k\ge 3$, there exists $\xi>0$ such that the following holds for all sufficiently large $n\in\N$. Let $F$ be a $k$-vertex graph and $G$ be an $n$-vertex graph with a partition $\mathcal{P}= \{V_1, V_2,\ldots,V_C\}$ for some integer $C$ such that for each $i\in[C]$, the following two properties hold:
\begin{enumerate}
  \item [$(1)$] for every vertex $v\in V_i$, there is an $F$-fan  at $v$ in $V_i$ of size at least $\ga n$;
  \item [$(2)$] every $S\in \binom{V_i}{k}$ has at least $\ga n$ vertex-disjoint $(F, t)$-absorbers in $G$.
\end{enumerate}
Then $G$ contains a family $\{A_1, A_2,\ldots, A_C\}$ of vertex-disjoint subsets of size at most $\tfrac\ga C n$ such that by letting $R=\bigcup_{i\in[C]}A_i$, for each $i\in [C]$,  $A_i$ is a $\xi$-absorbing set with respect to $V_i\setminus R$.
\end{lemma}

\begin{proof}
  Let $\ga,k,t, C$  be given and choose $0\ll\xi\ll\beta\ll q\ll \tfrac{\ga}{ktC}$. Furthermore fix $n$ sufficiently large, some  $G,F,\mathcal{P}$  as in the assumptions and for convenience, let $\gamma':=\tfrac\ga C$. We shall iteratively build our family $\{A_1, A_2,\ldots, A_C\}$ of vertex-disjoint subsets one at a time. At each step $s=1,\ldots,C$ we will choose $A_s$ of size at most $\ga' n$ such that $A_s$ is disjoint from $R_{s-1}$ and $A_s$ is a $\xi$-absorbing set with respect to $V_s\setminus R_{s-1}$ where $R_{s-1}=\cup_{1\leq i\leq s-1}A_i$ is the set of vertices already used in absorbing sets.  Note that $|R_{s-1}|\leq \gamma' n (C-1)$ and we have $\gamma n-|R_{s-1}|\geq \gamma' n$. So suppose we are at step $s\ge 1$, let $A_1,A_2,\ldots, A_{s-1}$ be the absorbing sets constructed so far  and $R_{s-1}$ as above (and so  $R_{s-1}=\emptyset$ if $s=1$). Further, let $V_s'=V_s\setminus R_{s-1}$. Now it suffices to find in $G-R_{s-1}$ a $\xi$-absorbing set $A_s$ with respect to $V_s'$  such that $|A_s|\le \ga' n$.   

Note that each vertex of $R_{s-1}$ can be in at most one of the $(F,t)$-absorbers given by $(2)$, as they are vertex disjoint. Hence every subset $S\in \binom {V_s'}{k}$ has at least $\ga n-|R_{s-1}|\ge \ga' n$ vertex-disjoint $(F, t)$-absorbers inside $G-R_{s-1}$ and similarly for every $v\in V_s'$ there is a fan $\cF_v$ at $v$ in $V_s'$ of size at least $\ga'n$. 
Consider a random subset $X$ obtained by independently including each vertex in $V_s'$ with probability $q$.
Since $\mathbb{E}[|X|]= q|V_s'|$ tends to infinity as $n$ increases, it follows from Chernoff's inequality (see e.g. \cite[Theorem 2.1]{janson}) that a.a.s.\   $\tfrac q 2 |V_s'|\leq|X|\leq 2q|V_s'|$. Also, for every $v\in V_s'$, let $f_v$ denote the number of the sets from $\F_v$ that lie inside $X$. Clearly,
  $\mu:=\mathbb{E}[f_v]\geq q^{k-1}|\F_v|\geq q^{k-1}\ga' n$. By a union bound and Chernoff's inequality, we have $ \mathbb{P}\left[\,\text{there is~}\, v\in V_s' \text{~with~} f_v<
      \frac{\mu}{2}\,\right]=o(1)$.
Therefore, as $n$ is sufficiently large, there is some $X\subset V_s'$ such that $ \tfrac{q\ga'}{2}n\leq|X|\leq 2qn$ and such that, for each $v\in V_s'$, there is a subfamily $\F_v'$ of at least
  $\tfrac{q^{k-1}\ga'}{2}n\ge 2\beta n$ sets from $\F_v$ contained in $X$.  This subset $X$ will form part of our absorbing set $A_s$ and will be used to find copies of $F$ containing vertices outside of $A_s$, when proving that $A_s$ is indeed an absorbing set.

Let $m := |X|/(1+\be)$ and note that $m$ is linear in $n$ and so we can assume that $m$ is sufficiently large. Let $B$ be the (auxiliary) bipartite graph  with vertex classes $X_m\cup Y_m$ and $Z_m$, obtained by applying Lemma \ref{robust}.
Arbitrarily choose vertex-disjoint subsets $Y,\,Z\<V_s'\setminus X$ with $|Y|=2m$ and $|Z|=3m(k-1)$, noting that $V_s'\setminus X$ is sufficiently large to do so. Now partition $Z$ arbitrarily into $3m$ subsets $\Z = \{Z_i\}_{i\in[3m]}$, each of size $k-1$, and fix bijections $\phi_1: X_m\cup Y_m \rightarrow X\cup Y$ (which is possible as the sets have the same size due to our definition of $m$) and $\phi_2: Z_m \rightarrow \Z$ such that $\phi_1(X_m) = X$ and $\phi_1(Y_m) = Y$.  

\begin{claim} \label{clm:absorbers}
There exists a family $\{A_e\}_{e\in E(B)}$ of pairwise vertex-disjoint subsets in $V(G')-(R_{s-1}\cup X\cup Y\cup Z)$ such that for every $e = \{w_1,w_2\} \in E(B)$ with $w_1\in X_m\cup Y_m$ and
$w_2\in Z_m$, the set $A_e$ is an $(F,t)$-absorber for $\phi_1(w_1)\cup \phi_2(w_2)$. 
\end{claim}
The idea here is to greedily choose such absorbers one by one for each $e\in E(B)$. Indeed, suppose we have already found appropriate subsets for all the edges from some  $E'\subset E(B)$ with~$E'\neq E(B)$.  Note that $m\leq2qn/(1+\be')$ and $\Delta(B)\leq 40$.  Therefore
    \[|X|+|Y|+|Z|+\bigg|\bigcup_{e\in E'}A_e\bigg| \leq
    4m+3m(k-1)+k^2t|E'|<4km+40k^2t|Z_m| \leq  \tfrac{\ga'}{2}n.\]
Since for each $e=\{w_1,w_2\} \in E(B)\setminus E'$, $S_e:=\phi_1(w_1)\cup \phi_2(w_2)$ is a $k$-set in $V_s'$ with at least $\ga' n$ vertex-disjoint $(F,t)$-absorbers in $G-R_{s-1}$, we can always choose one, say $A_e$, which is disjoint from $R_{s-1}\cup X\cup Y\cup Z\cup\bigcup_{e\in E'}A_e$. This establishes Claim \ref{clm:absorbers}. 

Let $A_s=X\cup Y\cup Z\cup\bigcup_{e\in E(B)}A_e$.  Then $|A_s|\le \ga' n$ and we claim that $A_s$ is a $\xi$-absorbing set with respect to $V_s'$. Indeed, take an arbitrary subset $U\subseteq V_s'\setminus A_s$ such that
$|U|+|A_s|\in k\mathbb{N}$ and $0\leq|U|\leq \xi n$. If there exists $Q\subset X$ with $|Q|=\be m$ and additionally an $F$-factor in $G[Q\cup U]$, then by setting
$X''=X\setminus Q$, which has size $m$, and $X_m'=\phi_1^{-1}(X'')$, Lemma \ref{robust} implies that
there is a perfect matching $M$ in $B$ between $X_m'\cup Y_m$ and $Z_m$. For each edge $e = \{w_1,w_2\}\in M$ take an $F$-factor in $G[\phi_1(w_1)\cup\phi_2(w_2)\cup A_e]$ and for each $e\in E(B)\setminus M$ take an $F$-factor in $G[A_e]$. This process  gives an $F$-factor of $G[A_s\setminus Q]$ and thus together with the $F$-factor in $G[Q\cup U]$ we get an $F$-factor in $G[A_s \cup U]$, as required.  

Hence, it remains to find such a set $Q$ as above. Note first that $(k-1)|U|\leq (k-1)\xi n\leq \be m$ due to our choice of constants.  We claim that $\be m-(k-1)|U|\in k\mathbb{N}$.  Indeed, taking disjoint subsets $Q_1,\,Q_2\subset X $ such that $|Q_1|=\be m-(k-1)|U|$ and $|Q_2|=(k-1)|U|$, we certainly have that
$|A_s|+|U|,\,|Q_2|+|U|\in k \mathbb{N}$ and, by virtue of the existence of
an $F$-factor in $G[A_s\setminus (Q_1 \cup Q_2)]$ (by the argument of the previous paragraph), we have that
$|A_s|-|Q_1|-|Q_2|\in k \mathbb{N}$.  Thus $|Q_1|=\be m-(k-1)|U|\in k\mathbb{N}$. Now we take an arbitrary subset $X'\subseteq X$ with
$|X'|=\left(\be m-(k-1)|U|\right)/k$.  Next we find a subset
$F_v\in \F_v'$ for each $v\in U\cup X'$ such that all these subsets $F_v$ are pairwise vertex-disjoint and do not contain any vertex of $X'$. In fact, we can choose such subsets greedily since $|X'|+(k-1)(|X'|+|U|)=\be m$ and
    $|\F_v'|\geq 
    2\be n> 2\be m$ as shown above.
Since $F_v\in\F_v'$ for every~$v\in U\cup X'$, there is an $F$-factor in~$G[Q\cup U]$ if we set $Q:= (\bigcup_{v\in U\cup X'}F_v)\cup X'$. Also, note that $|Q| =\be m$ as required.
This proves that $A_s$ is a $\xi$-absorbing set with respect to $V_s'$ and as we can find  such an $A_s$ for all $s$ throughout the process, this finishes the proof.
\end{proof}

\subsection{Finding fans} \label{sec:finding fans}

Lemmas \ref{bip-temp} and \ref{bip-temp+} reduce the problem of finding absorbing sets to finding absorbers and in the case of Lemma~\ref{bip-temp+}, large fans. In fact, finding fans is relatively easy due to our independence conditions on the host graph. For cycles and trees, we need the following well-known result due to Gy\'{a}rf\'{a}s, Szemer\'{e}di and Tuza \cite{gyarfas80} and independently, Sumner \cite{sumner81}.
\begin{lemma}\emph{\cite{gyarfas80,sumner81}}\label{universality}
  A $k$-chromatic graph contains every tree on $k$ vertices as a subgraph. Therefore for any graph $G=(V,E)$ and $k$-vertex tree $T_k$, if $U\subset V$ with $|U|> (k-1)\alpha(G)$ then $G[U]$ contains a copy of $T_k$.
\end{lemma}
As a corollary to Lemma \ref{universality}, we get the following which we will use throughout the proof.

\begin{cor}\label{sumner}
  For an integer $k$, an $n$-vertex graph $G$, a vertex $v\in V(G)$ and a vertex subset $U\subset V(G)$, the following holds. If $d_U(v)> (k-1)\alpha(G)$, then for any $k$-vertex $F$ which is a cycle $C_k$ or a tree $T_k$, there is a copy of $F$ containing $v$ and $k-1$ vertices of $U$.
\end{cor}

To derive Corollary \ref{sumner}  when $F=C_k$, we apply  Lemma \ref{universality} with $T$ being a path on $k-1$ vertices which then forms a copy of $C_k$ with $v$.  The following simple lemma shows that we can get large fans.

\begin{lemma} \label{lem:large fans}
Let $r,k\in \N$ with $k\geq 3$ and a $k$-vertex $F$ be such that one of the following holds:
\begin{enumerate}
 \item [$(1)$] $F=K_k$ and $r=k$ or $r=k-1$;
  \item [$(2)$] $F=T_k$ for some $k$-vertex tree or $F=C_k$ and $r=2$;
\end{enumerate}
Then for any $\rho>0$ there exists  $\alpha>0$ such that the following holds for all $n\in \N$. If $G$ is an $n$-vertex graph with $\alpha_r(G)\leq \alpha n$, $v\in V(G)$ and $U\subset V(G)$ are such that $d_U(v)\geq \rho n$ then there is an $F$-fan at $v$ in $U$ of size at least $\tfrac{\rho}{k}n$.
\end{lemma}

\begin{proof}
Choosing $0<\alpha\ll \tfrac{\rho}{k}$, this is essentially immediate. Indeed, take a fan $\cF_v$ at $v$ in $U$ of maximum size and suppose that $|\cF_v|<\tfrac{\rho}{k}n$. Then taking $N=(N(v)\cap U)\setminus V(\cF_v)$, we have that $|N|\geq \tfrac{\rho}{k}n$ at which point we can use that $\alpha_{k-1}(G)\leq \alpha n $ in case $(1)$ or Corollary~\ref{sumner} in case $(2)$ to find a further copy of $F$, extending $\cF_v$ and contradicting its maximality.
\end{proof}

\subsection{Finding absorbers}\label{sec:finding absorbers}

In order to apply Lemmas \ref{bip-temp} and \ref{bip-temp+}, what remains is to prove the existence of many absorbers.  In more detail, we need to  find a partition of the vertex set such that each part has the property that every $k$-set in the part has many absorbers.
In order to find this partition,  we adopt the  the latticed-based absorbing method developed in \cite{han16,han17,keevash15}. In this section we illustrate some of the key concepts of the method, introducing the relevant definitions  and stating the key lemmas that we will use to get our desired partition. We defer the proof of the major  lemmas (namely Lemmas \ref{lem:detectpartition} and \ref{lem:closed}) to  Section \ref{sec:lattice}.   


We will use the following notation introduced by Keevash and Mycroft
\cite{keevash15}.
Let $\mathcal{P} = \{V_1,\ldots,V_C\}$ be a partition of $V(G)$. For any subset $S\subset V(G)$, the \emph{index vector} of $S$ with respect to $\mathcal{P}$, denoted by $\textbf{i}_{\mathcal{P}}(S)$, is the vector in $\mathbb{Z}^r$ whose $i$th coordinate is the size of the intersections of S with $V_i$ for each $i\in[C]$. For $j \in [C]$, let $\textbf{u}_j\in \mathbb{Z}^r$ be the $j$th unit vector, i.e. $\textbf{u}_j$ has $j$th coordinate  $1$ and
 all other coordinates $0$. 
 A \emph{transferral} is a vector of the form $\textbf{u}_i-\textbf{u}_j$ for some
distinct $i\neq j\in[C]$. A vector $\textbf{i}\in \mathbb{Z}^r$ is an $s$-vector if all its coordinates are non-negative and their sum equals
$s$. Given $\be > 0$ and a $k$-vertex graph $F$, a $k$-vector $\textbf{v}$ is called $(F,\be)$-\emph{robust} if for any set $W$ of at most $\be n$ vertices, there is a copy of $F$ in $V(G)\setminus W$ whose vertex set has index vector $\textbf{v}$. Let $I^{\be}(\mathcal{P})$ be the set of all $(F,\be)$-robust $k$-vectors and $L^{\be}
(\mathcal{P})$ be the lattice (i.e. the additive subgroup) generated by $I^{\be}(\mathcal{P})$.  

We will also need the notion of $F$-reachability introduced by Lo and Markstr\"{o}m \cite{allan15}. Let 
$m,t$ be positive integers. Then we say that two vertices $u, v\in V(G)$ are $(F, m, t)$-\emph{reachable} (in $G$) if for any set $W\subset V(G)$ of at most $m$ vertices, there is a set $S\subset V(G)\setminus W$ of size at most $kt-1$ such that both $G[S\cup \{u\}]$ and $G[S\cup \{v\}]$ have $F$-factors, where we call such $S$ an $F$-\emph{connector} for $u,v$. Moreover, a set $U\<V(G)$ is $(F,m,t)$-\emph{closed} if every two vertices $u,v $ in $U$ are $(F,m,t)$-reachable. Note that for $U$ to be closed we do \emph{not} require that  the corresponding $F$-\emph{connectors} for $u,v$ are contained $U$. That is, the connectors may contain vertices in $V\setminus U$.    

The following result builds a sufficient condition on a given partition $\{V_1, V_2,\ldots,V_C\}$ to ensure that every $S\in \binom{V(G)}{k}$ of a certain type has linearly many vertex-disjoint absorbers.

\begin{lemma}\label{close}
Given $k, t\in \N$ with $k\ge 3$ and  $\be>0$, the following holds for any $k$-vertex graph $F$ and sufficiently large $n\in \mathbb{N}$. Let $G$ be an $n$-vertex graph with a partition $\mathcal{P}= \{V_1, V_2,\ldots,V_C\}$ for some integer $C\in \N$ such that each $V_i$ is $(F,\be n, t)$-closed with $|V_i|\ge \be n$ for each $i\in [C]$. If $S\in \binom{V(G)}{k}$ such that $\emph{\textbf{i}}_{\mathcal{P}}(S)$ is $(F,\be)$-robust, then $S$ has at least $\tfrac{\be}{k^3t}n$
vertex-disjoint $(F,t)$-absorbers.
\end{lemma}

\begin{proof}
For positive integers $k, t$ with $k\ge 3$ and  $\be>0$, let $G,F,\mathcal{P}$ and $C$ be given as in the assumption. For any $k$-subset $S\<V(G)$ with $\textbf{i}_{\mathcal{P}}(S)$ being $(F,\be)$-robust, we greedily construct as many pairwise disjoint $(F,t)$-absorbers for $S$ as possible. Let $\mathcal{A}=\{A_1,A_2,\ldots, A_{\ell}\}$ be a maximal family of $(F,t)$-absorbers constructed so far. Suppose to the contrary that $\ell<\tfrac{\be}{k^3t}n$.
Then $|\bigcup_{i=1}^\ell A_i| \le \tfrac{\be}{k} n $ as each such $A_i$ has size at most $k^2t$.   

By $(F,\be)$-robustness, we can pick a copy of $F$ inside $V(G)\setminus(\bigcup_{i=1}^\ell A_i\cup S)$ whose vertex set $T$ has the same index vector as $S$. Let $S=\{s_1,s_2,\ldots,s_k\}$ and $T=\{t_1, t_2,\dots, t_k\}$ such that $s_i$ and $t_i$ belong to the same part of $\mathcal{P}$ for each $i\in [k]$. We now greedily pick up a collection $\{S_1,S_2,\ldots, S_{k}\}$ of vertex disjoint subsets in $V(G)\setminus (\bigcup_{i=1}^\ell A_i\cup S\cup T)$ such that each $S_i$ is an $F$-connector for $s_i,t_i$ with $|S_i|\le kt-1$. Since
\[\left|\bigcup_{i=1}^\ell A_i\cup\left(\bigcup_{i=1}^{k'}S_i\right)\cup S\cup T\right|\le\be n,\]
for any $0\leq k'\leq k$ (using that $n$ is sufficiently large), we can pick each such $S_i$ one by one because $s_i$ and $t_i$ are $(F,\be n,t)$-reachable.
At this point, it is easy to verify that $\bigcup_{i=1}^{k}S_i\cup T$ is actually an $(F,t)$-absorber for $S$, contrary to the maximality of $\ell$.
\end{proof}

To apply Lemma~\ref{close}, we need to find a closed partition in order to guarantee the existence of many absorbers.
 The following crucial result provides such a partition $\mathcal{P}=\{V_1, V_2,\ldots,V_C\}$ of $V(G)$.

\begin{lemma}\label{lem:detectpartition}
Fix $k,r\in \N$ with $k\geq 3$, a $k$-vertex graph $F$, and  $\eta>0$ such that one of the following holds:
\begin{enumerate}
  \item [$(1)$] $F=K_k$, $r=k$ and $\eta>0$;
  \item [$(2)$] $F=K_k$, $r=k-1$ and $\eta\ge\frac{1}{k}$;
  \item [$(3)$] $F=T_k$ for some $k$-vertex tree, $r=2$ and $\eta>0$;
  \item [$(4)$] $F=C_k$, $r=2$ and $\eta\ge\frac{1}{k}$.
\end{enumerate}
Then for any  $D\in \N$ and $\mu>0$ there exists $\alpha,\be_0>0$ and $t_0\in \N$  such that the following holds for all sufficiently large $n\in \N$. If $G$ is an $n$-vertex graph with $\delta(G)\ge (\eta+\mu) n$ and $\alpha_r(G)\le \alpha n$ then there exists constants $\beta>0$ with $\be_0<\be< \mu^2$, and  $t\in \N$ with $t\le t_0$, and a partition $\mathcal{P}= \{V_1, V_2,\ldots,V_C\}$ of $V(G)$ for some integer $C\in \N$   such that  $V_i$ is $(F,\be n, t)$-closed, $|V_i|> \left(\eta+\tfrac{\mu}{2}\right)n$ and $e(V_i,V_j)< \frac{\be}{Dt} n^2$ for all $i\neq j\in [C]$.

\end{lemma}

We remark that Lemma~\ref{lem:detectpartition} shows much more than just the existence of a closed partition. Indeed, it gives
  the added condition that there are not many edges between parts. This property will be used to guarantee that the vectors $k\textbf{u}_i$ are $(F,\beta)$-robust, which we will also need to apply Lemma~\ref{close}. Moreover, each part in the partition is large i.e.\ of size at least $\left(\eta+\tfrac\mu 2\right)n$. This implies that the number of parts $C$ is less than $\floor{\tfrac{1}{\eta}}$ which is best possible due to the disjoint cliques construction discussed in Section \ref{sec:quasi}.   

Our last crucial lemma suggests that under the constraint of $\alpha_2^*(G)=o(n)$ instead of $\alpha_r(G)=o(n)$, we can  merge all the parts in the partition into only one part, which will be adopted in the proofs of Theorem \ref{thm:stareverything} \ref{RTT*1} and \ref{RTT*3}.

\begin{lemma}\label{lem:closed}
Fix $k\in \N$, a $k$-vertex graph $F$ and  $\eta>0$ such that one of the following holds:
\begin{enumerate}
  \item [$(1^*)$] $F=T_k$ for some $k$-vertex tree and $\eta>0$;
  \item [$(2^*)$] $F=C_k$, $k\geq 4$ and $\eta\geq\frac{1}{k}$.
\end{enumerate}
Then for any $\mu>0$, there exist $\alpha,\beta>0$ and $t\in \N$ such the following holds for sufficiently large $n$. Suppose $G$ is an $n$-vertex graph with $\de(G)\ge (\eta+\mu) n$ and $\alpha_2^*(G)\le \alpha n$. Then $V(G)$ is
 $(F,\be n, t)$-closed.

\end{lemma}

Proving Lemmas \ref{lem:detectpartition} and \ref{lem:closed} will be the subject of Section \ref{sec:lattice} after  building the relevant theory.


\subsection{Proof of the existence of absorbing sets}

 We close this section by proving Propositions \ref{prop:Absorbing sets star} and \ref{prop:Absorbing sets} taking Lemmas \ref{lem:detectpartition} and \ref{lem:closed} for granted for now. We begin with Proposition \ref{prop:Absorbing sets star} whose proof will now follow simply by putting together several lemmas. Let us spell out the details.

 \begin{proof}[Proof of Proposition \ref{prop:Absorbing sets star}]
 Fix $k\in \N$, $\eta>0$ and $F$ as in the statement of the proposition (so that $(1^*)$ or $(2^*)$ are satisfied). Furthermore, fix some $\mu, \gamma>0$ and choose $0\ll\alpha\ll \xi\ll \gamma'\ll\tfrac{1}{t},\beta\ll \gamma,\mu, \eta, \tfrac{1}{k}.$ Finally let $n\in \N$ be sufficiently large and fix an $n$-vertex graph $G=(V,E)$ with $\de(G)\ge (\eta+\mu)n$ and $\alpha_2^*(G)\le \alpha n$.

 By Lemma \ref{lem:closed} (and the fact that we choose $\alpha$ and $\beta$ sufficiently small and $t$ sufficiently large), we have that $V$ is $(F,\be n,t)$-closed. Now note that for any set $W$ of at most $\beta n$ vertices, there is a copy of $F$ in $V\setminus W$. Indeed this follows, for example, by Corollary~\ref{sumner}. Therefore, by Lemma~\ref{close} (applied with $C=1$), every $k$-set $S\in \binom{V}{k}$ has at least $\gamma' n$ vertex-disjoint  $(F,t)$-absorbers. Furthermore, applying Lemma~\ref{bip-temp} (with $\gamma'$ playing the role of $\gamma$) we have that $G$ contains a $\xi$-absorbing set of size at most $\gamma'n\leq \gamma n$, concluding the proof.
 \end{proof}

The proof of Proposition~\ref{prop:Absorbing sets} is similar but there is a little more work to do, to deal with atypical vertices. The details follow.

\begin{proof}[Proof of Proposition~\ref{prop:Absorbing sets}]
Fix $k,r\in \N$, a $k$-vertex $F$ and $\eta>0$ such that we fall into one of the cases $(1)-(4)$ as given in the statement of the proposition. Furthermore, let $\gamma,\mu>0$ and choose \[0<\alpha\ll\xi\ll \gamma' \ll \alpha',\beta_0,\tfrac{1}{t_0} \ll \tfrac{1}{D}\ll \epsilon \ll \mu,\gamma,\eta,
\tfrac{1}{k}.\] Finally 
let $n$ be sufficiently large and $G=(V,E)$ be an $n$-vertex graph with $\delta(G)\ge (\eta+\mu)n$ and $\alpha_r(G)\le \alpha n$.   

Now applying Lemma~\ref{lem:detectpartition}, we get some $C,t\in \N$ with $t\le t_0$,  $\be_0<\be<\mu^2$ and  a partition $\cP'=\{V'_1,\ldots,V'_C\}$ of $V$ such
that each $V_i'$ is $(F,\be n, t)$ closed and has  $|V'_i|\geq \left(\eta+\tfrac\mu 2\right)n$. Moreover for all $i\neq j\in [C]$ we have that $e(V_i',V_j')<\tfrac{\be}{Dt}n^2$. Note that the lower bound on the size of the $V'_i$ implies that $C\le \floor{\tfrac{1}{\eta}}$. Indeed if $C\geq \ceil{\tfrac{1}{\eta}}$, then we would have that $G$ has at least $C(\eta+\mu)n>n$ vertices a contradiction. This constant $C$ will remain fixed but we need to adjust the partition $\cP'$ to obtain the desired partition. First though, let us note that  $k\textbf{u}_i$ is $(F,\be)$-robust for every $i\in[C]$, i.e. $k\textbf{u}_i\in I^{\be}(\mathcal{P})$.  Indeed, this can be seen, for example, by Lemma~\ref{lem:large fans}; for any set $W$ of at most $\be n$ vertices, $G[V_i'\setminus W]$ has many edges and, by averaging, a vertex of large degree and therefore contains a copy of $F$. Therefore we can apply Lemma~\ref{close} to conclude that for any $i\in [C]$ and $S\in \binom{V'_i}{k}$, $S$ has a family $\A_S$ of at least $\tfrac{\be}{k^3t}n$ vertex-disjoint $(F,t)$-absorbers.

For each $i\in [C]$, let $Q_i=\bigcup_{j\neq i}V_j$ and $B_i=\{v\in V_i\mid d_{Q_i}(v)\ge \tfrac{\mu}{2} n\}$. Since $|B_i| \tfrac{\mu}{2}n\le e(V_i,Q_i)<\frac{C\be}{Dt} n^2$, we deduce that $|B_i|<\frac{\epsilon^2\beta}{t} n$ due to the fact that $\tfrac{1}{D}\ll \epsilon, \eta,\mu$.
Now for all $v\in \cup_{i\in [C]}B_i$, by Lemma~\ref{lem:large fans}, there is an $F$-fan at $v$ in $G$ of size at least $\tfrac{\eta}{k}n>k\sum_{i\in [C]}|B_i|$. Therefore, we can greedily find a minimal family (of size  of at most $\sum_{i\in [C]}|B_i|$)  of vertex-disjoint copies of $F$ that cover $\bigcup_{i\in[C]} B_i$. Let $\mathcal{B}$ be such an $F$-tiling and define $B=V(\cB)$ to be the vertices that feature in this tiling, noting that $|B|\le \tfrac{\epsilon \beta}{t}n$. Moreover, define  $V_i''=V_i'\setminus B$ for each $i\in[C]$ and $G''=G-B$.   

Now for any $i\in [C]$, and $S\in \binom{V_i''}{k}$, as we had a family $\A_S$ of at least $\tfrac{\beta}{k^3t}n$ vertex-disjoint $(F,t)$-absorbers in $G$ and at most $\tfrac{\epsilon\beta}{t}n$ of these intersect $B$, we have that $S$ has at least $\gamma'n$ vertex-disjoint absorbers in $G'$. Moreover for every $i\in [C]$ and vertex $v\in V_i''$, as $v\notin B$ we have that $d_{V_i''}(v)\ge \left(\eta+\tfrac\mu 2\right)n-|B|\ge \eta n$ and so by Lemma~\ref{lem:large fans} there is an $F$-fan at $v$ in $V_i''$ of size at least $\gamma' n$. Finally then,  applying Lemma~\ref{bip-temp+} in $G'$ we get a family $A_1,A_2,\ldots,A_C$ in $V\setminus B$ of vertex-disjoint subsets such that the following holds.  Taking $R=\cup_{i\in[C]}A_i$ and defining $V_i:=V_i''\setminus R$, we have that $A_i$ is a $\xi$-absorbing set with respect to $V_i$ for each $i\in [C]$. Defining $\cP=\{V_1,\ldots,V_C,A_1,\ldots,A_C,B\}$ it is easy to check that all the desired properties of the proposition hold and hence this concludes the proof.

\end{proof}

\section{Lattice-based absorbing method}\label{sec:lattice}

In this section we prove Lemmas~\ref{lem:detectpartition} and \ref{lem:closed}, thus completing the proof of Propositions~\ref{prop:Absorbing sets star} and \ref{prop:Absorbing sets}. The aim is to provide a partition of the vertex set so that each part is closed and there are few edges between parts. To begin with, in Section ~\ref{sec:closed partition}, we show that we can find some  closed partition into constantly many parts. This partition may not be optimal but it forms a basis for a process of merging parts to obtain the final partition. In Section~\ref{sec:transferrals}, we provide some auxiliary results which will be used for the merging process. In Section~\ref{sec:merge}, we then prove Lemma~\ref{lem:detectpartition} showing how the merging process terminates on a partition with the desired properties. Finally, we will prove Lemma~\ref{lem:closed}, showing that in the setting of $\RTT^*$, the optimal partition must be trivial, that is, simply $\{V(G)\}$.

\subsection{Finding a closed partition} \label{sec:closed partition}

Our first  crucial lemma  gives a closed partition provided every vertex is reachable to linearly many other vertices.

\begin{lemma}[Partition lemma]\label{partition}
For any constant $\de > 0$ and integer $k \ge 2$, there exist $\be>0$ 
and an integer $t>0$
such that the following holds for sufficiently large $n$. Let $G$ be an $n$-vertex graph and $F$ be a $k$-vertex graph. If every vertex in $V(G)$ is $(F,\de n, 1)$-reachable to at least $\de n$ other vertices, then there is a partition $\mathcal{P}= \{V_1, V_2,\ldots,V_C\}$ of $V(G)$ for some integer $C\le \ceil{\tfrac{1}{\de}}$ such that for each $i\in[C]$, $V_i$ is $(F,\be n, t)$-closed and $|V_i|\ge  \tfrac{\de }{2}n$.
\end{lemma}

We shall make use of the following subtle observation in the proof of Lemma \ref{partition}.
\begin{fact}\label{fact:concatenate}
Let $G$ be an $n$-vertex graph  and let $F$ be a graph of $k$ vertices.
For two vertices $u, w\in V(G)$, if there exist $m_1$ vertices $v\in V(G)$ which are $(F, m_2, t)$-reachable to both $u$ and $w$, respectively, then $u$ and $w$ are $(F, m, 2t)$-reachable, where $m= \min\{m_1-1, m_2-kt\}$.
\end{fact}

\begin{proof}[Proof of Lemma \ref{partition}]
Let $r_0 = \lceil1/\delta\rceil+1$ and  choose constants $0<\beta_{r_0+1}\ll \beta_{r_0}\ll\beta_{r_0-1} \ll \cdots \ll \beta_1=\de$. Furthermore, fix $\beta=\beta_{r_0+1}$ and $t=2^{r_0}$.
We can assume that there are two vertices that are not $(F, \beta_{r_0} n, 2^{r_0-1})$-reachable, as otherwise we just output
$\mathcal P=\{V(G)\}$ as the desired partition. 
Observe  also that every set of $r_0$ vertices must contain two vertices that are $(F, \beta_2n, 2)$-reachable to each other.
Indeed,
fixing some arbitrary set of vertices $S=\{v_1,\ldots,v_{r_0}\}\subset V(G)$,
by the Inclusion-Exclusion principle we have that there exists a pair $i\neq j\in [r_0]$ so that  both $v_i$ and $v_j$ are 
both $(F, \beta_1 n, 1)$-reachable to a set of at least $\delta' n$  vertices for some $\delta'\ge \tfrac{\de}{\binom{r_0}{2}}$.
Thus, by Fact~\ref{fact:concatenate}, we deduce that $v_i$ and $v_j$ are $(F, \beta_2 n, 2)$-reachable for as  $\beta_2\ll \beta_1=\delta$.  

To ease the notation, we write $\lambda_i=\beta_{r_0+2-i}$ and $t_i = 2^{r_0+1-i}$ for each $i\in [1,r_0+1]$. Therefore \[0<\lambda_1\ll \lambda_2\ll \lambda_3\cdots\ll\lambda_{r_0+1}=\be_1.\] Let $d$ be the largest integer with $2\le d \le r_0$  such that there are $d$ vertices $v_1,\dots, v_d$ in $G$ which are pairwise \emph{not} $(F, \lambda_d n, t_d)$-reachable. Note that $d$ exists and $2\le d\le r_0-1$. Indeed, we assumed above that there are  $2$ vertices that are not $(F,\lambda_2n,t_2)$-reachable and  if $d= r_0$, then there are $r_0$ vertices in $G$ which are pairwise not $(F, \beta_{2} n, 2)$-reachable, contrary to the observation above. Let $S=\{v_1, \dots, v_d\}$ be such a set of vertices and note that  we know that $v_1, \dots, v_d$ are also pairwise not $(F, \lambda_{d+1} n, t_{d+1})$-reachable, as $\lambda_{d+1}\leq \lambda_d$ and $t_{d+1}\geq t_d$.  

For a vertex $v$ we write $\tilde{N}_i(v)$ for the set
of vertices which are $(F, \lambda_i n, t_i)$-reachable to $v$. Consider $\tilde{N}_{d+1}(v_i)$ for $i\in [d]$.
Then we conclude that
\begin{description}
  \item[(i)] any vertex $v\in V(G) \setminus \{ v_1, \dots, v_d\}$ must be
in $\tilde{N}_{d+1}(v_i)$ for some $i\in [d]$.
Otherwise, $v, v_1, \dots, v_d$ are pairwise not $(F, \lambda_{d+1} n, t_{d+1})$-reachable, contradicting the maximality of $d$.
  \item[(ii)] $| \tilde{N}_{d+1}(v_i) \cap \tilde{N}_{d+1}(v_j)| \le \lambda_d n+1$ for any $i\not = j$. Otherwise Fact \ref{fact:concatenate} implies that $v_i, v_j$ are $(F, \lambda_d n, t_d)$-reachable to each other, a contradiction (using that  $\lambda_d\ll\lambda_{d+1}$ here).
\end{description}

For $i\in [d]$, let \[U_i = (\tilde{N}_{d+1}(v_i)\cup \{ v_i\})
\setminus \bigcup_{ j\not = i} \tilde{N}_{d+1}(v_j).\]
Then we claim that each $U_i$ is $(F, \lambda_{d+1}n, t_{d+1})$-closed.
Indeed otherwise, there exist $u_1, u_2\in U_i$ that are not
$(F, \lambda_{d+1}n, t_{d+1})$-reachable to each other. Then
$\{ u_1, u_2\} \cup \{ v_1, \dots, v_d\} \setminus \{v_i\}$
contradicts the maximality of $d$.  

Let $U_0 = V(G) \setminus (U_1\cup\dots\cup U_d)$.
We have $|U_0| \le d^2\lambda_d n$ due to {\textbf{(i)}} and  {\textbf{(ii)}} above. In order to obtain the desired reachability partition,
we now drop each vertex of $U_0$ back into some $U_i$ for $i\in [d]$ as follows. Since each $v\in U_0$ is $(F, \lambda_{r_0+1}n, t_{r_0+1})$-reachable to at least $\de n$ vertices by assumption, it holds due to the fact that  $\lambda_d\ll \de$, that
$|\tilde{N}_{r_0+1}(v)\setminus U_0| \ge \delta n - |U_0|\ge \delta n - d^2\lambda_d n > d\lambda_d n.$
Therefore there exists some $i \in [d]$ such that $v$ is $(F, \lambda_{r_0+1}n, t_{r_0+1})$-reachable to at least
$\lambda_d n+1$ vertices in $U_i$ and hence $(F, \lambda_{d+1}n, t_{d+1})$-reachable to all these vertices. Fact \ref{fact:concatenate} (using that $\lambda_d\ll\lambda_{d+1}$) implies  that  $v$ is $(F, \lambda_d n, t_d)$-reachable to every vertex in $U_i$.
Now partition $U_0$ as  $U_0=\cup_{i\in d}R_i$ where for each $i\in[d]$,  $R_i$ denotes a set of  vertices $v\in U_0$ that are $(F, \lambda_d n, t_d)$-reachable to every vertex in $U_i$. Again by Fact \ref{fact:concatenate}, for every $i\in [d]$,  every two vertices in $R_i$ are $(F, \lambda_{d-1} n, t_{d-1})$-reachable to each other. Let $\mathcal{P}=\{V_1, \dots, V_d\}$ be the resulting partition by setting $V_i=R_i\cup U_i$. Then each $V_i$ is $(F, \lambda_{d-1} n, t_{d-1})$-closed.
Also, for each $i\in[d]$, it holds that \[|V_i| \ge |U_i| \ge |\tilde{N}_{d+1}(v_i)| - d^2 \lambda_d n
\ge |\tilde{N}_{r_0+1}(v_i)| - \tfrac{\de}{2} n
\ge  \tfrac{\de}{2}n.\]
This completes the proof by noting that  $\be=\beta_{r_0+1}\le\lambda_{d-1}$ and $t=2^{r_0}\ge t_{d-1}$.
\end{proof}

Lemma~\ref{partition} gives us a closed partition if every vertex is reachable to a linear number of other vertices. It will be used in conjunction with the following lemma which shows that in our settings, this is indeed the case.

\begin{lemma}\label{reachable}
Fix $k,r\in \N$ with $k\geq 3$ and a $k$-vertex graph $F$ such that one of the following holds:
\begin{enumerate}
  \item [$(1)$] $F=K_k$ and $r=k-1$;
\item [$(2)$] $F=T_k$ for some $k$-vertex tree or $F=C_k$ and $r=2$.
\end{enumerate}
Then for any $\mu>0$ there exists $\alpha>0$  
such that the following holds for sufficiently large $n$. Let  $G$ be an $n$-vertex graph with $\de(G)\ge \mu n$ and $\alpha_r(G)\le \alpha n$.
Then every vertex in $G$ is $(F, \be n, 1)$-reachable to at least $\tfrac{\mu}{4} n$ other vertices, where $\be=\tfrac{\mu^2}{8}$.
\end{lemma}

\begin{proof}
Choose $0<\alpha\ll \mu,\tfrac{1}{k}$. 
For any vertex $u\in V(G)$, we write $\Gamma_{\be}(u)$ for the set of vertices $v$ such that $|N(u,v)|\ge 2\be n$. Then by double counting $e(N(u), V(G))$, we have \[|N(u)|\de(G)\le e(N(u), V(G))\le 2(|\Gamma_{\be}(u)||N(u)|+2\be n^2). \] Thus $|\Gamma_{\be}(u)|\ge \frac{\de(G)}{2}-\frac{2\be n^2}{|N(u)|}\ge \left(\frac{\mu}{2}-\frac{2\be }{\mu}\right)n\ge \frac{\mu}{4}n.$  

Now it suffices to show that $u$ is $(F, \be n, 1)$-reachable to every vertex $v$ in $\Gamma_{\be}(u)$ for both cases $(1)$ and $(2)$. Indeed, in all cases, for every such vertex $v\in \Gamma_{\be}(u)$ and an arbitrary set $W$ of at most $\be n$ vertices , we have $|N(u,v)\setminus W|\ge \be n> k\alpha n$. Thus in case $(1)$, by the assumption $\alpha_{k-1}(G)\le \alpha n$, we can find a $K_{k-1}$ in $N(u,v)\setminus W$, which induces a $K_k$-connector for $u,v$. In case $(2)$, Lemma \ref{universality} immediately gives a copy of  $T_{k-1}$ inside $N(u,v)\setminus W$ for any $k-1$ vertex tree $T_{k-1}$. Choosing $T_{k-1}$ by removing one vertex from $F$, we have that this copy of $T_{k-1}$ serves as an  $F$-connector for $u,v$. This completes the proof.
\end{proof}

\subsection{Transferrals} \label{sec:transferrals}

As in Lemma~\ref{lem:detectpartition}, we need a partition $\mathcal{P}=\{V_1, V_2,\ldots,V_C\}$ such that there are few edges between parts and consequently, appealing to Lemma~\ref{close},  every $S\in\binom{V_i}{k}$ has linearly many vertex-disjoint $(F,t)$-absorbers. Note that the current partition obtained from Lemma \ref{partition} might not
have this property. We proceed with a sequence of merging processes, which  follows the strategy of transferrals in \cite{han17}.  The following lemma builds a sufficient condition that allows us to merge two distinct parts into a closed one, recalling the definition of $I^{\be}(\mathcal{P})$ from Section \ref{sec:finding absorbers}.





\begin{lemma}[Transferral]\label{transferral}
Given any positive integers $k, t\in \mathbb{N}$  with $k\ge 3$ and  constant $\be>0$, the following holds for sufficiently large $n$. Let $F$ be a $k$-vertex graph and $G$ be an $n$-vertex graph with a partition $\mathcal{P}=\{V_1, V_2,\ldots,V_C\}$ of $V(G)$ such that each $V_i$ is $(F,\be n, t)$-closed. For distinct $i, j \in[C]$, if there exist two $k$-vectors $\textbf{s},\textbf{t}\in I^{\be}(\mathcal{P})$ such that $\textbf{s}-\textbf{t}=\textbf{u}_i-\textbf{u}_j$, then $V_i\cup V_j$ is $\left(F,\tfrac{\be}{2}n, 2kt\right)$-closed.
\end{lemma}

\begin{proof}
Let $C$, $\be>0$ and $\mathcal{P}=\{V_1, V_2,\ldots,V_C\}$ be a partition of $V(G)$ as in the assumption.
Moreover let $i\neq j\in [C]$ be such that there exists $k$-vectors $\textbf{s},\textbf{t}\in I^{\be}(\mathcal{P})$ such that $\textbf{s}-\textbf{t}=\textbf{u}_i-\textbf{u}_j$. Without loss of generality (relabelling if necessary) we may assume that $i = 1$ and $j = 2$. Then we may write $\textbf{s}=(s_1,s_2,\ldots, s_C)$ and  $\textbf{t}=(s_1-1,s_2+1,s_3,\ldots, s_C)$ for some $s_i\in \N$ such that $\sum_{i=1}^Cs_i=k$.  It suffices to show that every two vertices $x\in V_1$ and $y\in V_2$ are $\left(F,\tfrac{\be}2n,2kt\right)$-reachable so let us fix such an $x$ and $y$. Fix some vertex set $W\subset V(G)\setminus \{x,y\}$ of size at most $ \tfrac{\be}{2}n$. By the assumption of $(F,\be)$-robustness, we can pick two vertex-disjoint copies $F_1,F_2$ of $F$ in $V(G)\setminus (W\cup\{x,y\})$ whose corresponding vertex sets $S, T$ have index vectors $(s_1,s_2,\ldots, s_C)$ and $(s_1-1,s_2+1,\ldots, s_C)$, respectively.  

Note that we may choose $x'\in S\cap V_1$ and $y'\in T\cap V_2$ such that by letting $S\setminus \{x'\}=\{u_1,u_2,\ldots, u_{k-1}\}$ and $T\setminus \{y'\}=\{v_1,v_2,\ldots, v_{k-1}\}$, $u_j$ and $v_j$ belong to the same part of $\mathcal{P}$ for each $j\in [k-1]$. Since each $V_i$ is $(F,\be n, t)$-closed for each $i\in [C]$, we greedily pick a collection $\{S_1,S_2,\ldots, S_{k-1}\}$ of vertex disjoint subsets in $V(G)\setminus (W\cup S\cup T\cup\{x,y\})$ such that each $S_j$ is an $F$-connector for $u_j,v_j$ with $|S_j|\le kt-1$. Indeed, note that as $n$ sufficiently large we have that  $|W|\le \tfrac{\be n}{2} \le \be n- 2k^2t$ and so for any $k'\leq k-1$ we have
\[\left|\left(\bigcup_{j=1}^{k'}S_j\right)\cup W\cup S\cup T\cup\{x,y\}\right|\le \be n.\] We can therefore indeed pick the $S_j$ one by one because of the fact that $u_j$ and $v_j$ are $(\be n ,t)$-reachable. Similarly, we additionally choose two vertex-disjoint (from each other and all other previously chosen vertices) $F$-connectors, say $S_x$ and $S_y$, for $x,x'$ and $y,y'$, respectively. At this point, it is easy to verify that the subset $\hat S:=\bigcup_{i=1}^{k-1}S_i\cup S_x\cup S_y\cup S\cup T$ is actually an $F$-connector for $x,y$ with size at most $2k^2t-1$. For example if we want a perfect $F$-tiling in $G[\hat S \cup \{x\}] $ (leaving $y$ uncovered), we can take the perfect $F$-tilings in $G[S_x\cup\{x\}]$, $G[S_y\cup\{y'\}]$ and $G[S_j\cup\{v_j\}]$ for $j\in [k-1]$, as well as the copy $F_1$ of $F$ on $S$.   Therefore by definition, $x$ and $y$ are $\left(F,\tfrac{\be}{2}n, 2kt\right)$-reachable.
\end{proof}

To apply Lemma \ref{transferral} in the merging process, we shall make use of the following result which ensures the existence of a transferral.  We only care about the case when $F=C_k$ or $T_k$ here.  When $F$ is a clique, we have to reason a little further and we leave this to the proof of Lemma~\ref{lem:detectpartition}. Note also that Lemma \ref{v1v2} does not apply to $C_3$, we will handle $F=C_3=K_3$ with the other cliques.

\begin{lemma}\label{v1v2}
Let $k\in \N$ with $k\ge 3$ and $F$ be a $k$-vertex graph such that one of the following holds.
\begin{enumerate}
    \item[(1)] $F=T_k$ for some $k$-vertex tree;
    \item[(2)] $F=C_k$ with $k\ge 4$.
\end{enumerate}
Then for any  constants $\be>0$ and $C\in \mathbb{N}$, there exists $\alpha
>0$ such that the following holds for sufficiently large $n$.  Let $G$ be an $n$-vertex graph with $\alpha(G)\le \alpha n$ and $\mathcal{P}=\{V_1, V_2,\ldots,V_C\}$ a partition of $V(G)$ such that each $V_i$ has at least $2\be n$ vertices. Then for any distinct $i, j \in[C]$, if $e(V_i,V_j)\ge 2\be n^2$, then there exist two $k$-vectors $\textbf{s},\textbf{t}\in I^{\be}(\mathcal{P})$ such that $\textbf{s}-\textbf{t}=\textbf{u}_i-\textbf{u}_j$.
\end{lemma}




\begin{proof}
Choose $\alpha<\beta, \tfrac{1}{k}$
and let $G, \mathcal{P}$ be given as in the assumption. Furthermore, suppose that for some $i\neq j \in[C]$, we have $e(V_i,V_j)\ge 2\be n^2$. Without loss of generality (after relabelling if necessary) let $i=1$ and $j=2$. We will show that  there exist $k$-vectors $\textbf{s},\textbf{t}\in I^{\be}(\mathcal{P})$ such that $\textbf{s}-\textbf{t}=\textbf{u}_1-\textbf{u}_2$.  

 Given any vertex set $W$ of size at most $\be n$, let $V_i'=V_i\setminus W$, $i\in[2]$. It is easy to check that $|V_i'|\ge \be n$ and $e(V_1',V_2')\ge \be n^2$, which immediately implies the existence of some $x_2\in V_2'$ with $d_{V_1'}(x_2)\ge \be n\ge k\alpha n$. Corollary \ref{sumner} then gives a copy of $F$ inside $V_1'\cup V_2'$ with index vector $\textbf{s}=(k-1,1,0,\ldots,0)$.  

Now let $B=\left\{v\in V_2'\mid d_{V_1'}(v)\le  \tfrac{
\be}{2}n\right\}$ and $A=V_2'\setminus B$. Then it follows that $|A|\ge  \tfrac{\be}{2}n$ because $e(V_1',V_2')\ge\be n^2$. Choose a subset $S\subset A$ of size $\tfrac{4}{\be}$, which is possible as $n$ is sufficiently large. Since each vertex $a\in A$ has $d_{V_1'}(a)>  \tfrac{\be}{2}n$, the Principle of Inclusion-Exclusion implies that there exist two vertices $u,w\in S$ 
having at least $\tfrac{\be^2}{16} n\ge k \alpha n$ common neighbours in $V_1'$. We denote these common neighbours by $S'$. Note that Lemma \ref{universality} ensures inside $S'$, the existence of the path $P_{k-2}$ or the tree $T_{k-2}$ obtained by deleting two leaves from $T_k$, which together with $\{u,w\}$ induces a $C_k$ or $T_k$, respectively. We therefore obtain a copy of $F$ whose index vector is $\textbf{t}=(k-2,2,0,\ldots,0)$, and this ends the proof because $\textbf{s},\textbf{t}\in I^{\be}(\mathcal{P})$ and $\textbf{s}-\textbf{t}=\textbf{u}_1-\textbf{u}_2$.
\end{proof}






\subsection{Proofs of Lemmas \ref{lem:detectpartition} and \ref{lem:closed}} \label{sec:merge}

We are now in a position to prove Lemmas~\ref{lem:detectpartition} and \ref{lem:closed}.
We shall first prove Lemma \ref{lem:detectpartition} by taking the partition of the vertex set given by Lemma~\ref{partition} and running a merging process that terminates on our desired partition. The proof strategy of Lemma \ref{lem:closed} then  follows  the same scheme as that of Lemma \ref{lem:detectpartition} except that by taking $\alpha^*(G)=o(n)$ and $F=T_k$ or $C_{k}$ $(k\ge4)$,  we will be able to show that the final partition must in fact be the trivial partition with just one part equal to $V(G)$. To simplify the presentation, we shall give a short proof of Lemma \ref{lem:closed} at the end of this subsection, building on the proof of Lemma~\ref{lem:detectpartition}.


\begin{proof}[Proof of Lemma \ref{lem:detectpartition}.]

Fix $k,r\in \N$, a $k$-graph $F$  and $\eta>0$ such that one of the cases $(1)-(4)$ in the statement of the lemma hold. Furthermore, fix some $D\in \N$, $\mu>0$ and choose $0<\be',\tfrac{1}{t'}\ll \de\ll \mu,\eta, \tfrac{1}{k}$.  Now let $n$ be sufficiently large and fix an $n$-vertex graph $G=(V,E)$ with $\de(G)\ge (\eta+\mu)n$ and $\alpha_r(G)\le \alpha n$. Applying Lemma \ref{reachable}, we have that every vertex in $G$ is $(F,\de n,1)$-reachable to at least $\de n$ other vertices. We then apply Lemma \ref{partition}  to obtain a partition $\mathcal{P}=\{V_1, V_2,\ldots,V_{C'}\}$ of $V$ for some integer $C'\le \ceil{\tfrac{2}{\de}}$, where each $V_i$ is $(F,\be' n, t')$-closed and $|V_i|\ge  \tfrac{\de}{2}n$. Now let $c=\ceil{\tfrac{2}{\de}}$ and define sequences of constants $\{\tau_\ell\}_{\ell=0}^c$,  $\{\ze_\ell\}_{\ell=0}^{c}$ by setting  $\tau_0=t'$, $\ze_0=\be'$ and  \[\tau_{\ell+1}=2k\tau_\ell=(2k)^{\ell+1}t', \hspace{4mm}\ze_{\ell+1}=\frac{\ze_{\ell}}{4D\tau_\ell}=\frac{\be'}{(4Dt')^{\ell+1}(2k)^{\ell!}}\] for all $0\le \ell\le c-1$. Let $t_0=\tau_c$ and $\be_0=\ze_c$ and choose $0<\alpha\ll \be_0,\tfrac{1}{t_0}$.  

 At this point, we shall proceed by iteratively merging two distinct parts as follows.  
Given the initial partition $\mathcal{P}=\{V_1, V_2,\ldots,V_{C'}\}$ where each $V_i$ is $(F,\ze_0 n, \tau_0)$-closed, if there exists distinct $i,j\in [C']$ and two $k$-vectors $\textbf{s},\textbf{t}\in I^{2\ze_1}(\mathcal{P})$ such that $\textbf{s}-\textbf{t}=\textbf{u}_i-\textbf{u}_j$, then we merge $V_i,V_j$ into a new part. By applying Lemma \ref{transferral} (with $\be=2\ze_1$), we obtain that $V_i\cup V_j$ is $(F,\ze_1 n, 2k\tau_0)$-closed. By renaming the parts if necessary, the resulting partition $\mathcal{P}=\{V_1, V_2,\ldots,V_{C'-1}\}$ therefore satisfies that each $V_i$ is $(F,\ze_1 n, \tau_1)$-closed, and this finishes the first step.  

At step $\ell$ $(\ell\ge 2)$, anchoring at the current partition $\mathcal{P}= \{V_1,\ldots,V_{C'+1-\ell}\}$ where each $V_i$ is $(F, \ze_{\ell-1} n, \tau_{\ell-1})$-closed, whenever there exist distinct $i,j\in [C'+1-\ell]$ and two $k$-vectors $\textbf{s},\textbf{t}\in I^{2\ze_\ell}(\mathcal{P})$ such that $\textbf{s}-\textbf{t}=\textbf{u}_i-\textbf{u}_j$, we merge $V_i$ and $V_j$ into a new part. By applying Lemma \ref{transferral} (with $\be=2\ze_\ell$) we obtain that $V_i\cup V_j$ is $(F,\ze_\ell n, \tau_\ell)$-closed. By renaming if necessary, we end the $\ell$-th step with a new partition $\mathcal{P}= \{V_1,\ldots,V_{C'-\ell}\}$ where each $V_i$ is $(F,\ze_\ell n, \tau_\ell)$-closed. In this way, we continue until the procedure terminates after at most $C'-1<c$ steps.  

Suppose we end up with a final partition $\mathcal{P}= \{V_1,\ldots,V_{C'-s}\}$ for some integer $s<C'$, where each $V_i$ is $(F,\ze_s n, \tau_s)$-closed. Fix $C=C'-s$. If $C=1$, we are done by simply letting $\be=\ze_s>\be_0$ and $t=\tau_s\le t_0$.  So assume that $C>1$. By the definition of $C$, it follows that for any distinct $i,j\in [C]$, there are no pairs of $k$-vectors $\textbf{s},\textbf{t}\in I^{2\ze_{s+1}}(\mathcal{P})$ such that $\textbf{s}-\textbf{t}=\textbf{u}_i-\textbf{u}_j$. Now we proceed by investigating the structure of $\mathcal{P}$. Let $R$ be an auxiliary graph defined on $C$ vertices, say $V(R)=\{v_1,v_2,\ldots,v_{C}\}$, such that $v_i,v_j$ are adjacent if and only if $e(V_i,V_j)\ge 4\ze_{s+1} n^2$

\begin{claim}
$R$ is an empty graph.
\end{claim}
\begin{proof}
Suppose otherwise that there is an edge in $R$, say without loss of generality  $\{v_1,v_2\}\in E(R)$. If we are in the case where $r=2$ and  $F=T_k$ for some $k$-vertex tree $T_k$ or $C_k$ for $k\ge 4$ (this corresponds to cases $(3)$ and $(4)$ in the statement of Lemma~\ref{lem:detectpartition} minus the $C_3$ part of case $(4)$ which will be covered by the analysis of case $(2)$), then by applying Lemma \ref{v1v2} (with $\be=2\ze_{s+1}$), 
we obtain two $k$-vectors $\textbf{s},\textbf{t}\in I^{2\ze_{s+1}}(\mathcal{P})$ such that $\textbf{s}-\textbf{t}=\textbf{u}_1-\textbf{u}_2$, contrary to the assumption. If $F=K_k$ and $r=k$ (case $(1)$ of the lemma), then it is easy to verify that $\textbf{s}=\textbf{u}_1+(k-1)\textbf{u}_2\in I^{2\ze_{s+1}}(\mathcal{P})$, using for example Lemma~\ref{lem:large fans}, whilst $\textbf{t}=k\textbf{u}_2\in I^{2\ze_{s+1}}(\mathcal{P})$ by our upper bound on $\alpha_k(G)$. This again contradicts our definition of $C$.   

Thus it remains to consider the case $F=K_k$ and $r=k-1$ (Case $(2)$ in the statement of Lemma~\ref{lem:detectpartition}). In this case, we first show that $R$ has maximum degree at most $1$. Suppose otherwise that there are two adjacent edges in $R$, say $\{v_1,v_2\},\{v_1,v_3\}\in E(R)$ (and so in particular $C\geq 3$). Note that for any choice of $W\<V(G)$ of size $2\ze_{s+1}n$, by letting $V_i'=V_i\setminus W$, $i\in[3]$, we have $e(V_1',V_2')\ge 2\ze_{s+1} n^2, e(V_1',V_3')\ge 2\ze_{s+1} n^2$. By averaging, there exist two vertices $x_2\in V_2', x_3\in V_3'$ such that both $x_2$ and $x_3$ have at least $2\ze_{s+1} n\ge \alpha n$ neighbours in $V_1'$. Thus we can find two copies of $F$, whose index vectors are $\textbf{s}_2=(k-1,1,0,\ldots,0)$ and $\textbf{t}_2=(k-1,0,1,\ldots,0)$, respectively, and we are done because $\textbf{s}_2,\textbf{t}_2\in I^{2\ze_{s+1}}(\mathcal{P})$ and $\textbf{s}_2-\textbf{t}_2=\textbf{u}_2-\textbf{u}_2$, contradicting the definition of $C$.  

So assume that $\De(R)\le1$ and as before, that $\{v_1,v_2\}\in E(R)$. We will  again reach a final contradiction by finding two  $k$-vectors $\textbf{s},\textbf{t}\in I^{2\ze_{s+1}}(\mathcal{P})$ such that $\textbf{s}-\textbf{t}=\textbf{u}_1-\textbf{u}_2$. It is easy to observe (for example by averaging and appealing to Lemma~\ref{lem:large fans}) that $\textbf{u}_1+(k-1)\textbf{u}_2,(k-1)\textbf{u}_1+\textbf{u}_2\in I^{2\ze_{s+1}}(\mathcal{P})$, which forces that  $k\textbf{u}_1,k\textbf{u}_2\notin I^{2\ze_{s+1}}(\mathcal{P})$. Hence by the definition of $(F, 2\ze_{s+1})$-robustness, for each $i\in[2]$, there exists a subset $W_i$ of at most $2\ze_{s+1}n$ vertices such that $G[V_i\setminus W_i]$ does not contain any copy of $K_k$. It follows that every vertex $v\in V_i\setminus W_i$ has $d_{V_i\setminus W_i}(v)<\alpha n$ for $i\in [2]$. In this case, we show that $\textbf{t}=(k-2)\textbf{u}_1+2\textbf{u}_2\in I^{2\ze_{s+1}}(\mathcal{P})$, which suffices due to the fact that $\textbf{s}=(k-1)\textbf{u}_1+\textbf{u}_2\in I^{2\ze_{s+1}}(\mathcal{P})$. 
Now for any choice of $W\<V(G)$ of size $2\ze_{s+1}n$, for $i=1,2$ let \[V_i''=V_i\setminus (W_i\cup W),~Q=\bigcup_{3\le j\le C}V_j ~\text{and}~ B_i=\{v\in V_i''\mid d_Q(v)\ge \de^2 n\}.\] Since $C\le c=\ceil{\tfrac{2}{\de}}$  and
$|B_i|\de^2 n\le e(V_i'',Q)<4C\ze_{s+1} n^2$ for each $i\in[2]$,
it holds that $|B_i|<\de^2 n<|V_i''|$, using that $\ze_{s+1}<\ze_0\ll \de$. Furthermore, we have that for every $i\in[2]$ and $v\in V_i''\setminus B_i$, \begin{equation} \label{eq:useful}
     d_{V_{3-i}''\setminus B_{3-i}}(v)\ge (\eta+\mu) n-|W|-|W_{3-i}|-\alpha n-\de^2 n-|B_{3-i}|>\left(\tfrac{1}{k}+\tfrac{\mu}{2}\right) n.\end{equation}
Without loss of generality assume that $|V_1|\le |V_2|$ and so $|V_1|\le \tfrac{n}{2}$.
 For any fixed $v\in V_1''\setminus B_1$, by \eqref{eq:useful}, we can find a subset $S_v\subset N(v)\cap V_2''\setminus B_2$ be of size $k-1$ such that $\{v\}\cup S_v$ induces a copy of $K_k$. Then again using \eqref{eq:useful}, the fact that $|V_1''\setminus B_1|\le |V_1|\le \tfrac{n}{2}$ and the Principle of Inclusion-Exclusion, there must be two vertices $u,w\in S_v$ having at least $\de n$ common neighbours in $V_1''\setminus B_1$.  This immediately yields a copy of $K_k$ with index vector $\textbf{t}=(k-2)\textbf{u}_1+2\textbf{u}_2$ as required.   
\end{proof}

So let $s,C$ and $\cP=\{V_1,\ldots,V_C\}$ be the final partition of the merging process, as above. Fix  $\be=\ze_s>\beta_0$ and $t=\tau_s<t_0$.
Now we have that each $V_i$ is $(F,\be,t)$-closed and from the claim,  $e(V_i, V_j)<4\ze_{s+1} n^2=\tfrac{\be}{Dt} n^2$ for any distinct $i,j\in[C]$. Moreover, for each $i\in [C]$, let $Q_i=\bigcup_{j\neq i}V_j$ and $B_i=\left\{v\in V_i\mid d_{Q_i}(v)\ge \tfrac{\mu}{2} n\right\}$. Since $\ze_{s}\ll\de$, $C<c$ and $|B_i|\tfrac{\mu}{2} n\le e(V_i,Q)<4C\ze_{s+1} n^2$, we deduce that $|B_i|<\tfrac{\de}{2}n<|V_i|$. Thus for each $i\in[C]$ and $v\in V_i\setminus B_i$, $|V_i|>d_{V_i}(v)>\left(\eta+\tfrac{\mu}{2}\right)n$.  This concludes the proof of the lemma.  
\end{proof}

We continue by proving Lemma \ref{lem:closed}.

\begin{proof}[Proof of Lemma \ref{lem:closed}]
Fix $k\in \N$, a $k$-graph $F$  and $\eta>0$ such that one of the cases $(1^*)$ or $(2^*)$ in the statement of the lemma hold. Now, fixing $r=2$, note that we fall into case $(3)$ or $(4)$ of Lemma~\ref{lem:detectpartition}. Moreover as we have that $\alpha_r(G)\le 2\alpha_r^*(G)$ for any graph $G$,  we have that, by choosing $\alpha$ sufficiently small, the conclusion of Lemma~\ref{lem:detectpartition} holds. Notice that the conclusion of Lemma~\ref{lem:closed} follows from Lemma~\ref{lem:detectpartition} if $C=1$ for  all  graphs $G$ we are interested in. Our proof strategy thus follows the proof of Lemma~\ref{lem:detectpartition} and we simply need to prove that, with the added condition that $\alpha_2^*(G)\le \alpha n$, we have that $C=1$.   

So define all parameters as in the proof of Lemma~\ref{lem:detectpartition} and
suppose to the contrary that $C=C'-s\ge 2$. Take $\cP=\{V_1,\ldots, V_C\}$ to be the partition at the end of the merging process. 
For $i=1,2$, define $Q_i=\bigcup_{j\neq i}V_j$ and $B_i=\{v\in V_i\mid d_{Q_i}(v)\ge \de^2 n\}$. For each $i=1,2$, since $|B_i|\de^2 n\le e(V_i,Q_i)<4C\ze_{s+1} n^2$, we deduce that $|B_i|< \de^2 n \le|V_i|$ using that $\ze_{s+1}\ll \de$ and therefore every $v\in V_i\setminus B_i$ has $d_{V_i\setminus B_i}(v)\ge \mu n-\de^2n-|B_i|> \tfrac{\mu}{2} n$, using that $\de\ll \mu$.  Therefore,  we have that  $k\textbf{u}_i \in I^{2\ze_{s+1}}(\mathcal{P})$ for each $i\in [2]$ by Lemma~\ref{lem:large fans} for instance.  

Based on this, we reach a final contradiction by showing that $\textbf{u}_1+(k-1)\textbf{u}_2\in I^{2\ze_{s+1}}(\mathcal{P})$. In fact, for any choice of $W\<V(G)$ of size $2\ze_{s+1}n$, let $V_i'=V_i\setminus W$ for each $i\in [2]$. Then for any fixed subset $A_1\subset V_1'$ of size $\alpha n$, it follows by the definition of $\alpha^*(G)$ that $e(A_1,V_2'\setminus B_2)\ge|V_2'\setminus B_2|-\alpha n>\tfrac{\de}{4} n$.
Hence, there exists a vertex $a_1\in A_1$ that has a set $S$ of at least $\frac{\de}{4\alpha}$ neighbours in $V_2'\setminus B_2$. As $\alpha \ll \de$, it follows from the Inclusion-Exclusion  Principle and the fact that $S\subset V_2'\setminus B_2$, that there exist two vertices $u,w\in S\cap N(a_1)$ having a set $Y$ of at least $k\alpha n$ common neighbours in $V_2'\setminus B_2$. If we are in case $(1^*)$ and $F=T_k$ for some $k$-vertex tree, let $T_{k-1}$ be the graph obtained by removing a leaf $w$ from $F$ and let $T_{k-2}$ be the graph obtained by removing the unique neighbour of $w$ in $F$, from $T_{k-1}$. Then applying Lemma~\ref{universality} in $Y$ guarantees  a copy of $T_{k-2}$ in $Y$. Consequently, there is a copy of   $T_{k-1}$  inside $Y\cup\{u\}$ which together with $\{a_1\}$, forms a copy of $T_k$. If we are in case $(2^*)$ and $F=C_k$ for some $k\ge 4$, Lemma~\ref{universality} guarantees  a copy of  $P_{k-3}$  inside $Y$ which together with $\{a_1,u,w\}$, forms a copy\footnote{Note that the assumption that $k\ge 4$ is crucial for this argument to work.} of $C_k$. In either case we have a copy of $F$ with index vector $\textbf{u}_1+(k-1)\textbf{u}_2$, and thus $\textbf{t}=\textbf{u}_1+(k-1)\textbf{u}_2\in I^{2\ze_{s+1}}(\mathcal{P})$. As $\textbf{s}=k\textbf{u}_2\in I^{2\ze_{s+1}}(\mathcal{P})$, and $\textbf{s}-\textbf{t}=\textbf{u}_2-\textbf{u}_1$, this contradicts that the merging process has terminated, concluding the proof.

\end{proof}

\section{Almost perfect tilings}\label{sec:almost}

In this section we address Proposition~\ref{prop:almost factors}, showing that in all the settings of interest, we can always find an $F$-tiling covering all but some small linear number of vertices. Proposition~\ref{prop:almost factors} is split into cases $(1)-(4)$ depending on the different graphs $F$ as well as the independence and minimum degree conditions. In fact, for some of these cases the conclusion of Proposition~\ref{prop:almost factors} is  immediate. Indeed in case $(1)$ where $F=K_k$ and $r=k$, the independence condition $\alpha_r(G)\le \alpha n$ guarantees an $F$-tiling covering all but $\delta n$ vertices for any $0< \alpha \ll \delta$. Similarly, for case $(3)$ when $F$ is some $k$-vertex tree and $r=2$, a simple application of Lemma~\ref{universality} guarantees that any graph $G$ with $\alpha(G)\le \alpha n$ has a $T_k$-tiling covering all but $\de n$ vertices (again choosing $0<\alpha \ll \de$).  



Therefore it remains to prove case $(2)$ and case $(4)$  of Proposition~\ref{prop:almost factors}, for which we will use the powerful Szemer\'{e}di Regularity Lemma \cite{Szemeredi1978}. and an almost tiling theorem of Shokoufandeh and Zhao \cite{shokou03}. 
We first introduce some basic notation to state our main tools. Given a graph $G$ and $A$, $B$  disjoint subsets of $V(G)$, the \emph{density} between $A$ and $B$ is defined as
\[d(A,B)=\frac{e(A,B)}{|A||B|}.\]

\begin{defn}[$\epsilon$-regularity]
Let $\epsilon>0$. A pair $(A, B)$ of disjoint vertex-sets in $G$
is $\epsilon$-\emph{regular} if for every $X \< A$ and $Y \< B$ satisfying $|X|>\epsilon |A|, ~~|Y|>\epsilon |B|,$ we have $|d(X,Y)-d(A,B)|<\epsilon.$
\end{defn}




We will use the following form of the regularity lemma, as given in \cite{kom-sim}.

\begin{lemma}[Regularity Lemma-Degree Form]\emph{\cite[Theorem 1.10]{kom-sim}}\label{reg}
For every $\epsilon>0$ there exists $M=M(\epsilon)$ such that the following holds for any fixed real number $d \in [0, 1]$. Let $G=(V, E)$ be an $n$-vertex graph. Then there is a partition $\{V_0, V_1,\ldots, V_{\ell}\}$ of ~$V$ for some $\tfrac{1}{\epsilon}\le\ell<M$, and a spanning subgraph $G'$ of $G$ satisfying the following properties:
\begin{enumerate}
  \item [$(\textbf{R1})$] $|V_0|\le \epsilon n$ and there exists some $L\in \N$ with $L\le\lceil\epsilon n\rceil$ such that for all $1\leq i \leq \ell$ the cluster $V_i$ induces an independent set in $G'$ and $|V_i|=L$;
  \item [$(\textbf{R2})$] each $v\in V$ satisfies $d^{G'}(v)\ge d^{G}(v)-(d+\epsilon)n$;
  \item [$(\textbf{R3})$] every pair $(V_i, V_j)$, $1\le i < j\le \ell$, is $\epsilon$-regular in $G'$, with density either $0$
or greater than~$d$.
\end{enumerate}
\end{lemma}
For a partition as given by Lemma~\ref{reg}, we also define the \emph{reduced graph}, denoted by $R_d$, which is  defined on vertex set $[\ell]$ by declaring $ij\in E(R_d)$ if the corresponding pair $(V_i,V_j)$ has density at least $d$ in $G'$.   

We will also use some known embedding results. For an $r$-chromatic graph $H$ on $k$ vertices, we write $u= u(H)$ for the smallest possible size of a color-class over all proper $r$-colorings of $H$. The \emph{critical chromatic number} of $H$ is \[\chi_{cr}(H):=\frac{(r-1)k}{k-u}.\]
Here we shall make use of the following result by Shokoufandeh and Zhao \cite{shokou03}, which resolved a conjecture  of Koml\'{o}s \cite{kom00}. The theorem in this form can be derived directly from  \cite[Theorem 1.9]{kom00}. 

\begin{theorem}\emph{\cite{shokou03}}\label{H-tiling}
For any $k\in N $  and a $k$-vertex graph $H$ there exist integers $N=N(k)$ and $K=K(k)$ such that, for all $n\ge N$, if $G$ is any $n$-vertex graph with \[\de(G)\ge\left(1-\frac{1}{\chi_{cr}(H)}\right)n,\] then $G$ contains an $H$-tiling that covers all but at most $K$ vertices of $G$.
\end{theorem}

We will apply Theorem~\ref{H-tiling} in the case where $H=K_{1,k-1}$, the star with $k$ vertices. For this $H$ note that
 $\chi_{cr}(H)=\tfrac{k}{k-1}$. We are now ready to prove Proposition~\ref{prop:almost factors} cases $(2)$ and $(4)$.
\begin{proof}[Proof of Proposition~\ref{prop:almost factors} $(2)$ and $(4)$]
Fix $k,r\in \mathbb{N}$ with $k\ge 3$, $\eta\ge\tfrac{1}{k}$ and a $k$-vertex graph $F$ such that either $F=K_k$ and $r=k-1$ or $F=C_k$ and $r=2$. Furthermore, let $\mu,\de>0$ be given and choose $0<\alpha \ll \epsilon \ll \tfrac{1}{N} \ll \tfrac{1}{K}\ll d\ll \mu, \de, \tfrac{1}{k}$.  Let $n\in \N$ be sufficiently large and let $G=(V,E)$ be an $n$-vertex graph with $\de(G)\ge (\eta+\mu) n$ and $\alpha_r(G)\le \alpha n$.

Applying the Regularity Lemma (Lemma \ref{reg}),  we have a partition $\mathcal{P}=\{V_0, V_1,\ldots, V_{\ell}\}$ of $V(G)$ for some $\ell\in\mathbb{N}$ with $N\le\ell<M=M(\epsilon)$, and a spanning subgraph $G'$ of $G$ with the properties $(\textbf{R1})-(\textbf{R3})$ as in the statement of Lemma~\ref{reg}.
Let $R=R_d$ be the reduced graph  on vertex set $[\ell]$ defined by the resulting partition. Then a
straightforward counting argument shows that $R$ has minimum degree at least $\ell/k$. Indeed, for each fixed $v\in V_i$, it holds that \[d^G(v)-(d+\epsilon)n\le d^{G'}(v)\le d^R(i)L+|V_0|,\] where the first inequality follows from property $(\textbf{R2})$ and the last inequality from $(\textbf{R1})$ and $(\textbf{R3})$. By the assumption $\de(G)\ge\frac{n}{k}+\mu n $, our upper bound on $|V_0|$ in $(\textbf{R1})$ and the choice of $d,\epsilon$, it follows that \[\de(R)\ge \frac{\frac{n}{k}+\mu n-(d+2\epsilon)n}{L}\ge \frac{n}{kL}\ge \frac{\ell}{k}.\]
Therefore, we can apply Lemma \ref{H-tiling}  (using that $\ell\ge N=N(k)$) and conclude  that for  $H=K_{1,k-1}$, $R$ contains a $H$-tiling  $\mathcal{M}=\{H_1,H_2,\ldots,H_s\}$ that covers all but at most $K$ vertices in $R$. Note that each copy $H_i$ of $H=K_{1,k-1}$ in $R$ corresponds to a collection $\{V_{i1},V_{i2},\ldots,V_{ik}\}$ of $k$ clusters  in the original partition $\mathcal{P}$, such that each pair $(V_{i1},V_{ij})$ is $\epsilon$-regular with density at least $d$ for $j\in[2,k]$. We will now show how to greedily pick  vertex-disjoint copies of $F$ that cover almost all the vertices in $\bigcup_{j=1}^{k} V_{ij}$. For convenience, we shall illustrate the general idea by simply taking sets $V_{1},V_{2},\ldots,V_{k}$ with $V_j=V_{1j}$ for $j\in [k]$.   

Now we proceed by partitioning $V_1$ arbitrarily into $k-1$ subsets $U_{2}, U_{3}, \ldots, U_{k}$ of  equal size (or as close to equal size as possible).  We claim that for each $2\le i\le k$, $G[U_i\cup V_i]$ admits an $F$-tiling covering all but at most $k\epsilon L$ vertices in $U_i\cup V_i$. Indeed, each such copy of $F$ inside $G[U_i\cup V_i]$ may be chosen greedily such that it contains exactly one vertex inside $U_i$. Suppose that we have already chosen such an $F$-tiling $\mathcal{I}_i$ inside $U_i\cup V_i$ that leaves at least $k\epsilon L+1$ vertices uncovered. Taking $U_i'=U_i\setminus V(\mathcal{I}_i)$ and $V_i'=V_i\setminus V(\mathcal{I}_i)$, we observe that $|U_i'|\ge \epsilon L$, $|V_i'|\ge (k-1)\epsilon L$. Since $(V_1,V_i)$ is an $\epsilon$-regular pair, it follows that $d(U_i',V_i')\ge d-\epsilon$ and therefore some vertex $x\in U_i'$ has at least \[(d -\epsilon)|V_i'|>\frac{d}{2}(k-1)\epsilon L>k\alpha n\]
neighbours inside $V_i'$, due to our choice of constants. Thus by the assumption that  $\alpha_{r}(G)\leq \alpha n$ (using Corollary \ref{sumner} for the case that $F=C_k$), there exists a subset $S_x\< V_i'$ of $k-1$ vertices, which together with $x$ yields a copy of $F$ that can be added to $\mathcal{I}_i$.  

From the above (and the fact that we can run the same process with $V_j=V_{ij}$ for all $i\in [s]$), we find an $F$-tiling for $G$ that leaves at most $|V_0|+KL+s(k-1)k\epsilon L\le\epsilon n+K\epsilon n+(k-1)\epsilon n\le\de n$  vertices uncovered, where the last inequality follows from the choice of $\epsilon$. This completes the proof of Proposition \ref{prop:almost factors}.
\end{proof}


\section{Constructions}\label{5}

In this section, we shall give short proofs of Theorem \ref{thm:RTTeverything} \ref{RTTii} and \ref{RTTv}  as well as Theorem~\ref{thm:stareverything} \ref{RTT*2} and \ref{RTT*4}. We first introduce  a simple dense graph that will guarantee our minimum degree conditions.

\begin{defn} \label{def:G0}
For $n\in N$ and  $\eta>0$ let $G_0=G_0(\eta)$ be an $n$-vertex graph with a vertex partition $V(G_0)=X_1\cup X_2$ such that
\begin{enumerate}
  \item[(i)] $|X_1|=\eta n$ and $G_0[X_1]$ is a clique,
  \item[(ii)] $G_0[X_1, X_2]$ is a complete bipartite graph, and
  \item[(iii)] $G_0[X_2]$ is an independent set.
\end{enumerate}
In other words, $G_0$ is simply an $n$-vertex clique with a clique of size $(1-\eta)n$ (on vertex set $X_2$) removed to leave an independent set.
\end{defn}

Next we give some  sparse probabilistic constructions with desired properties, where we use $g(G), $  $\omega(G)$ to denote the length of a shortest cycle and the order of a largest clique in $G$, respectively. The first construction is essentially due to Erd\H{o}s~\cite{erdos1959graph} and can be proven easily using the first moment method and alterations. 

\begin{lemma}\label{G1}
For any constants $\alpha>0$, $k\in \mathbb{N}$ with $k\ge3$, for all sufficiently large $n$, there exists an $n$-vertex graph $G_1$  such that $\alpha^*(G_1)<\alpha n$ and $g(G_1)>k$.

\end{lemma}


The next lemma can be shown by similar methods to Lemma~\ref{G1}, namely using the first moment method and alterations. Indeed, taking $p=\Theta(n^{-x})$ for some $\tfrac{2}{r+1}<x<\tfrac{2}{r}$, the random graph  $G(n,p)$ has $o(n)$ copies of $K_{r+1}$ and a copy of $K_{r}$ in every linear sized set, with high probability.
There is also an explicit construction of  Erd\H{o}s and Rogers \cite{erdos-rogers} which achieves the same properties.

\begin{lemma}\label{G2}
For any constant $\alpha>0$, $r\in \mathbb{N}$ with $r\ge2$, there exists an $n$-vertex graph $G_{\ell}$ for sufficiently large $n$ such that $\alpha_{r}(G_r)<\alpha n$ and $\omega(G_r)<r+1$
\end{lemma}

 We are now in a position to prove Theorem \ref{thm:RTTeverything}~\ref{RTTii} and \ref{RTTv} and Theorem \ref{thm:stareverything} \ref{RTT*2} which all follow the same scheme. For Theorem \ref{thm:RTTeverything}~\ref{RTTii}  we will in fact prove the following more general proposition giving lower bounds for $\RTT_r(K_k)$.

\begin{prop} \label{prop:general lower}
For $r,k\in \N$ with $1\leq r\le k-1$ and $k\ge 3$, $\RTT_r(K_k)\ge 1-\tfrac{r}{k}$.
\end{prop}
\begin{proof}
Fix $1\le r\le k-1$ as in the statement and let $\alpha>0$ and $0<\eta <1-\tfrac{r}{k}$ be given. Let $n$ be sufficiently large and define $\mu=1-\tfrac{k\eta}{k-r}>0$. We will show that there exists an $n$-vertex graph $G$ with $\delta(G)\ge \eta n$ and $\alpha_r(G)\le \alpha n$ such that every $K_k$-tiling in $G$ covers at most $(1-\mu)n$ vertices.

Indeed  let $G_0=G_0(\eta)$ be the $n$-vertex graph with vertex partition $V(G_0)=X_1\cup X_2$ as given by Definition~\ref{def:G0} and let $G_r$ be the $n$-vertex graph with $\alpha_{r}(G_r)<\alpha n$ and $\omega(G_r)<r+1$ given by Lemma~\ref{G2} (if $r=1$ simply let $G_r$ be the empty graph on $n$ vertices). Finally, let  $G=G_0\cup G_r$ which we claim has the desired properties.
Indeed it easy to see that  $\de(G)\ge \de(G_0)\ge  \eta n$ and $\alpha_{r}(G)\le \alpha_{r}(G_r)<\alpha n$.  Finally, as $\omega(G_r)<r+1$, every copy of $K_k$ must intersect $X_1$ in at least $k-r$ vertices. Thus every $K_k$-tiling in $G$ contains at most $\tfrac{|X_1|}{k-r}=\tfrac{\eta}{k-r} n$ vertex-disjoint copies of $K_k$, which cover at most $\tfrac{k\eta}{k-r} n=(1-\mu)n$ vertices.
\end{proof}

Our proof of Theorem \ref{thm:RTTeverything} \ref{RTTv} and Theorem~\ref{thm:stareverything} \ref{RTT*2} are similar, using the graph $G_1$ from Lemma~\ref{G1} instead of $G_r$. We omit the details here and
end the section by giving a short proof of Theorem \ref{thm:stareverything} \ref{RTT*4} which states that $\RTT^*_2(K_3)\geq\frac{1}{2}$.  
We will use so-called (near)-Ramanujan graphs. For an $n$-vertex $d$-regular graph $G$, we let $\lambda(G)$ denote the \emph{second eigenvalue} of $G$ which is defined to be $\lambda(G)=\max\{|\lambda_2|,|\lambda_n|\}$ where $\lambda_1\ge\lambda_2\ge\ldots\ge\lambda_n$ is the spectrum of the adjacency matrix of $G$. 
We will use the following corollary to a result of Friedman \cite[Theorem 1.1]{Fri08} on random $d$-regular graphs, which  confirms the existence of \emph{near}-Ramanujan graphs for every even $d\ge4$.
\begin{theorem}\emph{\cite{Fri08}}\label{Friedman}
For any even $d \ge 4$ and $\mu > 0$, there exists an $n_0\in \N$ such that for all $n\ge n_0$,  there exists a $d$-regular $n$-vertex graph $G_0$ with $\lambda(G_0)\le 2\sqrt{d-1}+\mu$.
\end{theorem}
Now we will give a construction for the lower bound as follows.
\begin{lemma} \label{lem:lastconst}
For any $0<\alpha<1$, the following holds for all  sufficiently large integers $n\in3\mathbb{N}$, letting $d=2\lceil(\frac{4}{\alpha})^4\rceil$. There exists an $n$-vertex graph $G$ with $\delta(G)\ge \frac{n}{2}-2d^2$ and $\alpha^*(G)\le \alpha n$ such that $G$ contains no $K_3$-factor.
\end{lemma}
\begin{proof}
Here we define a bipartite graph as follows: Let $n\in 3\N$ be sufficiently large and let $n_0=\ceil{\tfrac{n}{2}}+1$. Furthermore let  $G_0=(V,E)$ be an $n_0$-vertex $d$-regular graph such that $\lambda(G_0)\le2\sqrt{d-1}+\alpha$, whose existence is guaranteed by Theorem \ref{Friedman}. Then we take two identical copies $V_1$ and $V_2$ of $V$, and join $u\in V_1$ and $v\in V_2$ if $uv\in E(G_0)$. The resulting bipartite $d$-regular graph is denoted by $G^*$. Note that by the well-known Expander Mixing Lemma (see \cite[Theorem 2.11]{KS06}), $G^*$ satisfies that for any $A\<V_1,B\<V_2$  each of size $\alpha n_0$,
$
e(A,B)\ge \frac{d}{n_0}(\alpha n_0)^2-(2\sqrt{d-1}+\alpha)\alpha n_0>0.
$

Let $G'$ be obtained from $G^*$ by adding all possible edges inside each $V_i$ whilst maintaining that $S$ is an independent set of size $d$ whenever $S$ is the set of all neighbours of some vertex in $G^*$. Note that by the definition, $G'$ does not contain any triangles  whose vertex set intersects both $V_1$ and $V_2$. We also have  that $\delta(G')\ge n_0-1-d^2$ because every vertex is nonadjacent to at most $d^2$ vertices in the same part. Indeed, if $u$ and $v$ lie in the same part $V_i$ and are non-adjacent, then we must have that $u$ has distance 2 to $v$ in $G_0$.  In addition, we can easily check that $\alpha^*(G')\le \alpha n_0+d^2$.

Now if $n\in 6\N$ we have that  $2n_0=n+2$ whilst if $n\notin 6\N$ we have that $2n_0=n+4$. In either case deleting 2 (if $n\in 6\N$) or 4 (if $n\notin 6\N$) vertices  from $V_1$, we have that the resulting graph $G$ has $n$ vertices, $\alpha^*(G)\le \alpha n$ and $\de(G)\ge \de(G^*)-4\ge \tfrac{n}{2}-2d^2$.
Moreover, $G$ contains no $K_3$-factor. Indeed, there are no triangles with vertices in both $V_1$ and $V_2$ as previously noted and so any $K_3$-factor in $G$ must induce a $K_3$-factor in $G[V_1]$ and $G[V_2]$. However as $|V_1|\neq |V_2| \mod 3$ (due to how we deleted vertices from $G^*$), we have that this is impossible.

\end{proof}


\section{Concluding Remarks} \label{sec:conclude}

In this paper, we have explored the Ramsey--Tur\'an tiling functions $\RTT$ and $\RTT^*$, providing several upper and lower bounds for different regimes and graphs $F$.
There are many questions that remain open and a general understanding for all graphs $F$ seems challenging.   

For cycles and trees, however, we have close to the complete picture. Indeed for $k$-vertex trees $T_k$, we have shown that already $\RTT_2(T_k)=0$ and thus $\RTT_r(T_k)=0$ for all $r\geq 2$. With regards to $\RTT_1(T_k)$, the answer seems to depend on the structure of the tree. Indeed Koml\'os~\cite{kom00} proved that $\RTT_1(T_k)\ge
\left(1-\tfrac{1}{\chi_{cr}(T_k)}\right)$ (see Section~\ref{sec:almost} for the defintion of $\chi_{cr}$) whilst Shokoufandeh and Zhao~\cite{shokou03} proved that when $\de(G)\ge \left(1-\tfrac{1}{\chi_{cr}(T_k)}\right)n$ one can find a $T_k$-tiling covering all but a constant number of vertices, although the constant given is perhaps not optimal and does not guarantee a quasiperfect tiling. For $k$-vertex cycles, our results (see Theorem~\ref{thm:RTTeverything})  show that $\RTT_3(C_k)=0$ and $\RTT_2(C_k)=\tfrac{1}{k}$. Abbasi~\cite{abbasi1998solution} proved that $\RTT_1(C_k)=\tfrac{1}{2}$ when $k$ is even and $\RTT_1(C_k)=\tfrac{k+1}{2k}$ when $k$ is odd.   

For cliques, we make the following conjecture.

\begin{conj} \label{conj}
For $1\le r\le k\in \N$ with $k\ge 3$, $\RTT_r(K_k)=1-\tfrac{r}{k}$.
\end{conj}

The lower bounds of Conjecture~\ref{conj} follow from Proposition~\ref{prop:general lower} and for various values of $r$, the conjecture is established. Indeed, in this paper, we have shown Conjecture~\ref{conj} for $r=k$ and $r=k-1$ (see Theorem~\ref{thm:RTTeverything}). When $r=1$, Conjecture \ref{conj} is   the Hajnal-Szemer\'edi Theorem  \cite{hs} and the case $r=2$  (and $k\ge 4$) of Conjecture~\ref{conj} was recently shown by Knierim and Su~\cite{knier21}. Note that for $k=3,4$, we thus have the full picture for $\RTT(K_k)$ but for larger values of $k$, intermediate values for $r$ (namely, $3\le r\le k-2$) remain open. It would be very interesting to close this gap.  

For $\RTT^*$, the results for trees and cycles of length at least $4$ match those discussed above for $\RTT$. For cliques, Nenadov and Pehova~\cite{nenadov20} proved that $\RTT^*_k(K_k)=0$ and the results of Knierim and Su~\cite{knier21} and Hajnal and Szemer\'edi~\cite{hs} show that $\RTT^*_2(K_k)=1-\tfrac{2}{k}$ (for $k\ge 4)$ and $\RTT^*_1(K_k)=1-\tfrac{1}{k}$, respectively, as with $\RTT$. However it follows from our work here and that of Balogh, Molla and Sharidzadeh~\cite{balog16} that $\RTT^*_2(K_3)=\tfrac{1}{2}$. This separates $\RTT$ and $\RTT^*$ and it would be interesting to explore the behaviour of $\RTT^*$ in the intermediate range for larger cliques.   

Finally, we remark that  it would be  interesting to take a closer look at the boundary values where we see a jump in the behaviour. For example, our construction shows that when we impose that $\alpha_{k-1}(G)=o(n)$ and $\de(G)\ge \eta n$ for some $0<\eta<\tfrac{1}{k}$ it is possible to have graphs $G$ where any $K_k$-tiling leaves a linear number of vertices uncovered but when $\eta> \tfrac{1}{k}$ the number of uncovered vertices falls to a constant number. The transition when $\eta=\tfrac{1}{k}-o(1)$ remains mysterious.   \medskip

{\textbf{Acknowledgements:}}
We are very grateful to Andrew Treglown for enlightening discussions at the beginning of this project and for his  involvement with \cite{han2019tilings}, the proof methods of which had a strong bearing on this work. We are also indebted to Theo Molla, who suggested this problem to us. Finally, we would like to thank the anonymous referees for their thorough reading of the manuscript and many helpful suggestions.

\bibliographystyle{abbrv}
\bibliography{ref}
\end{document}